\documentclass[11pt]{article}
\usepackage{amsmath,amsthm, amssymb, amsfonts,accents,nccmath,mathrsfs}
\usepackage{enumitem}
\usepackage{array}
\usepackage[toc,page]{appendix}
\usepackage[english]{babel}
\usepackage{dsfont}
\usepackage{booktabs}
\usepackage[linesnumbered,commentsnumbered,ruled,vlined]{algorithm2e}
\usepackage{tikz}
\usepackage{tikz-network}
\usepackage[utf8]{inputenc}
\usepackage[english]{babel}

\usepackage{caption}
\usepackage{subcaption}

\usepackage{hyperref}
\hypersetup{
    colorlinks=true,
    linkcolor=blue,
    filecolor=darkmagenta,      
    urlcolor=cyan,
    citecolor=red
} 
\makeatletter
\tikzset{
    dot diameter/.store in=\dot@diameter,
    dot diameter=2pt,
    dot spacing/.store in=\dot@spacing,
    dot spacing=13pt,
    dots/.style={
        line width=\dot@diameter,
        line cap=round,
        dash pattern=on 0pt off \dot@spacing
    }
}
\makeatother

\makeatletter
\def\BState{\State\hskip-\ALG@thistlm}
\makeatother
\numberwithin{equation}{section}

\newtheorem{theorem}{Theorem}[section]

\newtheorem{lem}[theorem]{Lemma}
\newtheorem{prop}[theorem]{Proposition}
\newtheorem{defn}[theorem]{Definition}
\newtheorem{rem}[theorem]{Remark}

\newtheorem{obs}[theorem]{Observation}

\newcommand{\R}{{\mathbb R}}

\newenvironment{itemizeReduced}{
\begin{list}{\labelitemi}{\leftmargin=1em}
\setlength{\itemsep}{1pt}
\setlength{\parskip}{0pt}
\setlength{\parsep}{0pt}}{\end{list}
}

\newcommand{\ds}{\displaystyle}

\newcommand{\cof}{\textup{cof }}

\title{Distance Matrix of Weighted Cactoid-type Digraphs}
\author{Joyentanuj Das$^*$   \quad and \quad Sumit Mohanty\footnote{School of Mathematics, IISER Thiruvananthapuram, Maruthamala P.O., Vithura, 
Thiruvananthapuram,\newline \indent   Kerala- 695 551, India.
 \newline \indent Emails:
joyentanuj16@iisertvm.ac.in,  \quad sumit@iisertvm.ac.in, sumitmath@gmail.com }}
\date{}

\begin{document}

\maketitle

\begin{abstract}
A strongly connected digraph is called a cactoid-type if each of its blocks is a digraph consisting of finitely many oriented cycles sharing a common directed path. In this article, we find the formula for the determinant of the distance matrix for weighted cactoid-type digraphs and find its inverse, whenever it exists. We also compute the determinant of the distance matrix for a class of unweighted and undirected graphs consisting of finitely many cycles, sharing a common path.

\end{abstract}



\section{Introduction and Motivation}
 A graph $G=(V, E)$ is said to be a directed graph (or digraph) if every edge $e \in E$ is assigned with an orientation; otherwise, $G$ is called an undirected graph. For $x, y \in V$, we write, $x\to y$ to indicate the directed edge from   $x$ to  $y$. A graph $G$ is said to be a mixed graph if edges in a non-empty proper subset of  $E$ are assigned with orientations and in this case, the edges in $G$ which are not oriented can be considered as bidirected edges. A digraph or a  mixed graph is strongly connected if a directed path connects each pair of vertices. Distance well-defined graphs consist of connected undirected graphs, strongly connected digraphs and strongly connected mixed graphs, for more details see~\cite{Zhou1}.

Let $G=(V, E)$ be a distance well-defined graph. The distance between the vertices $x \in V$  and $y \in V$,  denoted as $d(x,y)$ equals the length (number of (directed) edges) of the shortest (directed) path from $x$ to $y$. Note that $d(x,x)=0$ for all $x\in V.$  Then, $D(G)=[d(x,y)]$ is an $|V|\times |V|$ matrix,  called the distance matrix of   graph  $G$.

A vertex $v$ of a graph $G$ is a cut-vertex of $G$, if $G - v$ is disconnected. A block of the graph $G$ is a maximal connected subgraph of $G$ that has no cut-vertex. Now we recall the definition of generalized distance matrix which is a generalization of distance matrix, for details see~\cite{Zhou1,Zhou3,Zhou2}.
\begin{defn}[Generalized Distance Matrix] Let $G=(V,E)$ be a distance well-defined graph. A generalized distance matrix  of $G$ is an   $|V|\times |V|$ matrix with $[D_{xy}]_{x,y \in V}$ satisfying the following conditions:
\begin{enumerate}
\item $D_{xx}=0$  for all $x\in V,$ and

\item  if  $x$ and $y$ are two vertices of $G$ such that every shortest (directed) path from $x$ to $y$ passes through the cut-vertex $v$, then $D_{xy}=D_{xv}+D_{vy}$.
\end{enumerate}  
\end{defn}

A weighted graph  $G^{W}$  is defined to be a  pair $(G,W)$, where  $G = (V, E)$ is the underlying graph of $G^{W}$ and $W: E\rightarrow \R$ is  the weight function, which assigns to each edge $e\in E$ with a real number $W(e)$, called the weight of $e$. In the case that all of the weights are equal to $1$, we refer to $G$ as an unweighted graph. Note that, the weights on edges do not reflect the structure of the graph.

Let $G=(V,E)$ be a distance well-defined graph and  $W: E\rightarrow \R$ be a weight function. For the weighted graph $G^W$, we define the total weight of a (directed) path $P$ in  $G^W$,  the sum of the weights on edges of $P$, and the total weight represents the number of edges on $P$ when $G$ is unweighted. For vertices $x, y \in  V$, let $d_{xy}$ equals to the minimum total weight   of a (directed) path from $x$ to $y$, and  we set $d_{xx} = 0$  for all $x\in V$. Let us define $D(G^W)=[d_{xy}]$. Observe that, if $G$ is unweighted, then the matrix $D(G^W)$ coincides with the distance matrix $D(G)$ defined earlier in this section.  For  more general cases:  
\begin{enumerate}
\item[\it{1}.] It is known that if the weight on each edge is positive, then $d_{xy}=d(x,y)$ defines a metric on the graph $G$ and $D(G^W)$ is called the distance matrix of $G$ with respect to the weight function $W$.
\item[\it{2}.]  It is easy to see that if the weights on the edges are real-valued (not necessarily positive), then $D(G^W)$ is a generalized distance matrix of $G$.
\end{enumerate}
The proofs for most of the results in this article are valid for both of the above cases unless specified otherwise. Thus, for the sake of convenience, we simply call $D(G^W)$ as the distance matrix of $G^W$  and also use  $D$  to represent $D(G^W)$ if there is no scope of confusion.

We now introduce some notations. Let $I_n$ and $ \mathds{1}$  denote the identity matrix and the column vector of all ones, respectively. Further, $J_{m \times n}$  denotes the $m\times n$ matrix of all ones and if $m=n$, we use the notation $J_m$. We will use $\mathbf{0}$ to represent zero matrix if there is no scope of confusion with the order of the matrix. Given a matrix $A$, we use $A^T$ to denote the transpose of the matrix.

Let $G=(V,E)$ be an undirected and unweighted graph. The adjacency matrix $A(G)$  of $G$ is an   $|V|\times |V|$ matrix with  $A(G) = [a_{xy}]_{x,y \in V}$, where \vspace*{-.2cm}
$$
a_{xy}=
\begin{cases}
1 &  \text{if }  x \mbox{ is adjacent to } y,\\
0 & \text{otherwise}.
\end{cases}
$$
The Laplacian matrix of  $G$,  defined as  $L(G)= \delta(G)-A(G)$, where $\delta(G)=$diag$(A\mathds{1})=$ diag$(\delta_x)$ and $\delta_x$ denotes the degree of the vertex $x$. It is well known that $L(G)$ is a symmetric and positive semi-definite matrix. The constant vector  $\mathds{1}$ is the eigenvector of $L(G)$ corresponding to the smallest eigenvalue $0$, and the Laplacian matrix satisfies $L(G)\mathds{1} = \mathbf{0}$ and $\mathds{1}^t L(G) = \mathbf{0}$ (for details see~\cite{Bapat}).

Let $T$ be a (unweighted and undirected) tree on vertices $1,2,\ldots,n$. In~\cite{Gr1}, the authors proved that the determinant of the distance matrix $D(T)$ of $T$ is given by $\det D(T)=(-1)^{n-1}(n-1)2^{n-2}.$ Thus the determinant does  not depend on the structure of the tree but only on the number of vertices. In \cite{Gr2}, it was shown that the inverse of the distance matrix of a tree is given by $D(T)^{-1} = -\dfrac{1}{2}L(T) + \dfrac{1}{2(n-1)}\tau \tau^T,$ where $\tau = (2-\delta_1,2-\delta_2,...,2-\delta_n)^T.$  The above expression gives formula for the  inverse of distance matrix of a tree in terms of the Laplacian matrix. Several extensions and generalization of this result have been studied in~\cite{Bp2,Bp3,JD1,Hou1,Hou2,Hou3,Zhou1,Zhou3,Zhou2}. The  primary objective of these extensions is to define a matrix  $\mathcal{L}$ called Laplacian-like matrix   satisfying  $\mathcal{L}\mathds{1} = \mathbf{0}$ and $\mathds{1}^t \mathcal{L} = \mathbf{0}$ and  find the inverse of the distance  matrix $D$  in the following form
\begin{equation}\label{eqn:inverse-form}
D^{-1} = -\mathcal{L} + \frac{1}{\lambda}\beta \alpha^T,
\end{equation}
where $\alpha$ and $\beta$ are column vectors and $\lambda$ is a suitable constant. Here in we want to mentioned that  except for the weighted cactoid digraph in~\cite{Zhou2}, $\alpha=\beta$ for all other cases  listed above.

Before proceeding further, we first recall the definition of cactoid digraph. A graph is said to be a cactoid digraph if each of its blocks is an oriented cycle. In literature, an undirected cycle with $n$ vertices is denoted by $C_n$, and for the oriented case it is denoted as $dC_n$. 

Let $G$ be a  graph consisting of finitely many  cycles sharing a common path $P$ such that the intersection between any two cycles is also precisely the path $P$. If the path $P$ is of length $n$ and $G$ consists of $r$ cycles of length $n+m_1+1,n+m_2+1,\ldots,n+m_r+1$, then we denote $G$ by $C(n;m_1,\cdots,m_r)$. If $G$ consists of $r$ oriented cycles and the common path $P$ is of length $n$ such that $G$ is strongly connected ({\it{i.e.},} the orientation on path $P$ agrees with all the  $r$ cycles), then we denote $G$ by $dC(n;m_1,\cdots,m_r)$, for example see Figure~\ref{fig:123}. The choice of parameters involved in the above notation is explained in Section~\ref{sec:Oriented Block}.
\begin{figure}[ht]
     \centering
     \begin{subfigure}{.33\textwidth}
         \centering
\begin{tikzpicture}[scale=0.4]
\Vertex[x=-3,y=1,label=1,position=left,size=.2,color=black]{1}
\Vertex[x=-1,y=3,label=2,position=above,size=.2,color=black]{2}
\Vertex[x=2,y=3,label=3,position=above,size=.2,color=black]{3}
\Vertex[x=4,y=1,position=right,size=.2,color=black]{4}
\Vertex[x=4,y=-1,position=right,size=.2,color=black]{5}
\Vertex[x=2,y=-3,position=below,size=.2,color=black]{6}
\Vertex[x=-1,y=-3,label=n-1,position=below,size=.2,color=black]{n-1}
\Vertex[x=-3,y=-1,label=0,position=left,size=.2,color=black]{n}

\Edge[label=$$,position=left,Direct](n)(1)
\Edge[label=$$,position={left=1mm},Direct](1)(2)
\Edge[label=$$,position={above},Direct](2)(3)
\Edge[position=below,style={dashed},Direct](3)(4)
\Edge[position=below,style={dashed},Direct](4)(5)
\Edge[position=below,style={dashed},Direct](5)(6)
\Edge[position=below,style={dashed},Direct](6)(n-1)
\Edge[label=$$,position={left=1mm},Direct](n-1)(n)
 
\end{tikzpicture}
         \caption{ $dC_n$}
     \end{subfigure}% <- nötig
       \begin{subfigure}{.3\textwidth}
         \centering
         \begin{tikzpicture}[scale=0.5]
\Vertex[x=0,y=0,size=.2,color=black]{1}
\Vertex[x=0,y=3,size=.2,color=black]{10}
\Vertex[x=0,y=6,size=.2,color=black]{2}
\Vertex[x=1,y=3,size=.2,color=black]{3}
\Vertex[x=2,y=3,size=.2,color=black]{4}
\Vertex[x=3,y=5,size=.2,color=black]{6}
\Vertex[x=3,y=3,size=.2,color=black]{7}
\Vertex[x=3,y=1,size=.2,color=black]{8}

\Edge(1)(10)
\Edge(10)(2)
\Edge(2)(3)
\Edge(3)(1)
\Edge(2)(4)
\Edge(4)(1)
\Edge(2)(6)
\Edge(1)(8)
\Edge(6)(7)
\Edge(7)(8)
\end{tikzpicture}
         \caption{$C(2;1,1,3)$}
     \end{subfigure}
     \begin{subfigure}{.3\textwidth}
         \centering
        \begin{tikzpicture}[scale=0.5]
\Vertex[x=0,y=0,size=.2,color=black]{1}
\Vertex[x=0,y=3,size=.2,color=black]{10}
\Vertex[x=0,y=6,size=.2,color=black]{2}
\Vertex[x=1,y=3,size=.2,color=black]{3}
\Vertex[x=2,y=4,size=.2,color=black]{4}
\Vertex[x=2,y=2,size=.2,color=black]{5}
\Vertex[x=3,y=5,size=.2,color=black]{6}
\Vertex[x=3,y=3,size=.2,color=black]{7}
\Vertex[x=3,y=1,size=.2,color=black]{8}

\Edge[Direct](1)(10)
\Edge[Direct](10)(2)
\Edge[Direct](2)(3)
\Edge[Direct](3)(1)
\Edge[Direct](2)(4)
\Edge[Direct](4)(5)
\Edge[Direct](5)(1)
\Edge[Direct](2)(6)
\Edge[Direct](8)(1)
\Edge[Direct](6)(7)
\Edge[Direct](7)(8)
\end{tikzpicture}
         \caption{$dC(2;1,2,3)$}
     \end{subfigure}
     \caption{}\label{fig:123}
\end{figure}

 We  define cactoid-type digraphs as follows: A strongly connected digraph is said to be cactoid-type, if each of its blocks is a  digraph $dC(n;m_1,\cdots,m_r)$, where $n,r\geq 1$ and $m_j\geq 1$; $1\leq j\leq r$. 
 
\begin{figure}[ht]
\centering
\begin{tikzpicture}[scale=1.1]
\Vertex[x=0,y=0,label=$$,position=left,size=.2,color=black]{1}
\Vertex[x=0,y=1,label=$$,position=left,size=.2,color=red]{2}
\Vertex[x=0,y=2,label=$$,position=left,size=.2,color=black]{3}
\Vertex[x=1,y=1.5,label=$$,position=right,size=.2,color=black]{4}
\Vertex[x=1,y=0.5,label=$$,position=right,size=.2,color=black]{5}
\Vertex[x=2,y=1.5,label=$$,position=above,size=.2,color=red]{6}
\Vertex[x=2,y=0.5,label=$$,position=left,size=.2,color=red]{7}
\Vertex[x=2,y=-0.5,label=$$,position=left,size=.2,color=black]{8}
\Vertex[x=3,y=0,label=$$,position=right,size=.2,color=black]{9}
\Vertex[x=4,y=0,label=$$,position=right,size=.2,color=black]{10}
\Vertex[x=3,y=1.5,label=$$,position=right,size=.2,color=black]{11}
\Vertex[x=3,y=2.5,label=$$,position=right,size=.2,color=black]{12}
\Vertex[x=-1,y=2,position=right,size=.2,color=black]{13}
\Vertex[x=-1,y=0,position=right,size=.2,color=black]{15}
\Vertex[x=-2,y=1,position=right,size=.2,color=black]{14}

\Edge[label=$$,position={left=1mm},Direct,fontcolor=red](1)(2)
\Edge[label=$$,position={left=1mm},Direct,fontcolor=red](2)(3)
\Edge[label=$$,position={left=1mm},Direct,fontcolor=red](3)(4)
\Edge[label=$$,position={left=1mm},Direct,fontcolor=red](4)(5)
\Edge[label=$$,position={left=1mm},Direct,fontcolor=red](5)(1)
\Edge[label=$$,position={left=1mm},Direct,fontcolor=red](6)(7)
\Edge[label=$$,position={left=1mm},Direct,fontcolor=red](3)(6)
\Edge[label=$$,position={left=1mm},Direct,fontcolor=red](7)(1)
\Edge[label=$$,position={left=1mm},Direct,fontcolor=red](7)(8)
\Edge[label=$$,position={left=1mm},Direct,fontcolor=red](8)(9)
\Edge[label=$$,position={left=1mm},Direct,fontcolor=red](9)(7)
\Edge[label=$$,position={left=1mm},Direct,fontcolor=red](8)(10)
\Edge[label=$$,position={left=1mm},Direct,fontcolor=red](10)(7)
\Edge[label=$$,position={left=1mm},Direct,fontcolor=red](6)(12)
\Edge[label=$$,position={left=1mm},Direct,fontcolor=red](12)(11)
\Edge[label=$$,position={left=1mm},Direct,fontcolor=red](11)(6)
\Edge[position={left=1mm},Direct,fontcolor=red](2)(13)
\Edge[position={left=1mm},Direct,fontcolor=red](13)(14)
\Edge[position={left=1mm},Direct,fontcolor=red](14)(15)
\Edge[position={left=1mm},Direct,fontcolor=red](15)(2)
\end{tikzpicture}
\captionof{figure}{Cactiod-type digraph with blocks $dC(1;2), dC(2;2,2), dC(1;1)$ and $dC(1;1,1)$}
\end{figure}
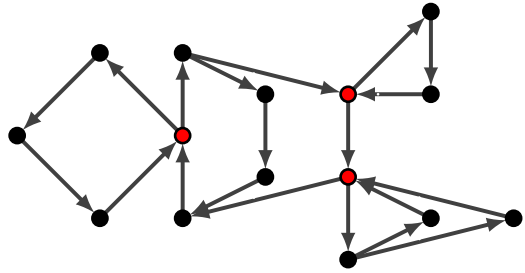

\begin{rem}
\begin{itemizeReduced}
\item[1.] It is easy to see, if $r=1$, then $C(n;m_1,\cdots,m_r)$ (or $dC(n;m_1,\cdots,m_r)$) represent a cycle
         (or oriented cycle) of certain length. So the notion of cactoid-type digraphs is an extension 
         to the cactoid digraphs.
\item[2.]If $n=1$, the $r$ cycles in a block shares  a common edge or common cut edge. So, 
            we are interested in graphs with a  common cut edge( if $n>1,$ it is a common path).
\end{itemizeReduced}
\end{rem}

In this article, we are interested in computing the determinant of  the distance matrix for weighted cactoid-type digraphs and find its inverse, whenever it exists. Due to results in~\cite{Gr3,Zhou2}, it is sufficient 
to find the determinant and the inverse (in the requisite form, see Eqn.~\eqref{eqn:inverse-form}) for individual blocks, the weighted digraph $dC(n;m_1,\cdots,m_r)$, denoted by $dC(n;m_1,\cdots,m_r)^W$, and  these sufficient conditions are discussed  in Section~\ref{subsec:pre-result}. 

In this search, we found that the Laplacian-like matrix $\mathcal{L}$ involved in the inverse of the distance matrix for the weighted digraph $dC(n;m_1,\cdots,m_r)^W$, is a ``perturbed weighted Laplacian"  for the same. Interestingly, the adjacency matrix involved in the weighted Laplacian for the digraph is not dependent on the weights on the individual edge, but only on the total weight of each of the $r$ cycles involved. It allows us to choose the weight on edge to be $0$, without disturbing the structure.

This article is organized as follows. In  Section~\ref{subsec:pre-result}, we summarized   few existing results useful for this article. In Sections~\ref{sec:Oriented Block} and~\ref{sec:inv-single}, we find the determinant and cofactor of the  distance matrix $D(G)$
 for weighted digraph $dC(n;m_1,\cdots,m_r)^W$ and compute its inverse, whenever it exists.  Consequently, we extend these results to cactoid-type digraphs  in Section~\ref{sec:cac-type}. Finally, in Section~\ref{sec:UnOriented Block}, we  compute the determinant  of the distance matrix for a  class of unweighted and undirected graph  $C(n;m_1,\cdots,m_r)$.

\section{Some Preliminary Results}\label{subsec:pre-result}
In this section, we recall some existing literature on the distance matrix,  that allows us to extend the results for determinant and inverse from individual blocks to the full graph. 

Given a matrix $A$, we use the notation $A(i \mid j)$ to denote the submatrix of $A$ obtained by deleting the $i^{th}$ row and the $j^{th}$ column from $A$.  Let $A$ be an $n \times n$ matrix. For $1 \leq i,j \leq n$, the cofactor $c_{i,j}$ is defined as $(-1)^{i+j} \det A(i \mid j)$. We use the notation  $\cof A$ to denote the sum of all cofactors of $A$. 

\begin{lem}\label{Lem:cof}\cite{Bapat}
Let $A$ be an $n \times n$ matrix. Let $M$ be the matrix obtained from $A$ by subtracting the first row from all other rows and then subtracting the first column from all other columns. Then 
$$\cof A = \det M(1 | 1).$$
\end{lem}

The result below gives the determinant and cofactor of the distance matrix of strongly connected digraph based on the determinant and cofactor of its blocks. It was first proved for distance matrix on strongly connected digraphs with non-negative weights,  for details see~\cite{Gr3}, but it is known that the result applies equally well for generalized distance matrix. 

\begin{theorem}\label{Thm:cof-det}\cite{Gr3}
Let $G$ be a strongly connected digraph with blocks $G_1,G_2,\ldots,G_b$. Then\vspace*{-.3cm}
$$\cof D(G) = \displaystyle \prod_{i=1}^{b} \cof D(G_i),$$ \vspace*{-.5cm}
$$\det D(G) = \displaystyle \sum_{i=1}^{b} \det D(G_i) \prod_{j \neq i} \cof D(G_j).$$
\end{theorem}

 Now using our notations, we summarize few definitions and results in~\cite{Zhou1,Zhou3} and~\cite[Section~3]{Zhou2}, that gives the sufficient condition to compute the inverse of the distance matrix for weighted cactoid-type digraph.

\begin{lem}\cite{Zhou2}\label{lem:inv}
Let $D$ and $\mathcal{L}$ be two $n \times n$ matrices, $\alpha$ and $\beta$ be two $n \times 1$ column vectors and $\lambda$ be a non-zeo number. If one of the following two conditions holds: %%%\vspace*{-.2cm}
\begin{itemize}
\item[1.] $\alpha^T D = \lambda \mathds{1}^T$ and $\mathcal{L}D + I = \beta  \mathds{1}^T$,
\item[2.] $D \beta = \lambda  \mathds{1}$ and $D\mathcal{L} + I =  \mathds{1} \alpha^T$,
\end{itemize}
then $D$ is invertible and $D^{-1} = -\mathcal{L} + \dfrac{1}{\lambda} \beta \alpha^T$.
\end{lem}
\begin{defn}
Let $D$ and $\mathcal{L}$ be two $n \times n$ matrices, $\alpha$ and $\beta$ be two $n \times 1$ column vectors and $\lambda$ be a number. Then
%%\vspace*{-.2cm}
\begin{itemize}
\item[1.]  $D$ is a left LapExp($\lambda, \alpha, \beta, \mathcal{L}$) matrix,  if  
$\alpha^T \mathds{1},\ \mathcal{L} \mathds{1}=0, \ \alpha^T D=\lambda \mathds{1}^T \mbox{ and } \mathcal{L}D+I=\beta \mathds{1}^T,$
\item[2.]  $D$ is a right LapExp($\lambda, \alpha, \beta, \mathcal{L}$) matrix,  if $\beta \mathds{1}^T=1,\  \mathds{1}^T \mathcal{L}=0,\ D\beta=\lambda\mathds{1}   \mbox{ and } D\mathcal{L} +I = \mathds{1}\alpha^T$.
\end{itemize}
\end{defn}

An $n$-bag is a tuple ($A,\lambda, \alpha, \beta, \mathcal{L}$) consisting of $n \times n$ matrices $A$ and $\mathcal{L}$, a number $\lambda$ and
$n \times 1$ column vectors $\alpha$ and $\beta$. An $n$-bag ($A,\lambda, \alpha, \beta, \mathcal{L}$) is called a left (or right) LapExp 
$n$-bag if $A$ is a left (or right) LapExp($\lambda, \alpha, \beta, \mathcal{L}$) matrix. Let $M$ be an $n \times n$ matrix whose
rows and columns are indexed by vertices of a graph $G$, and let $H$ be a subgraph of $G$. We use $\mathsf{sub}(M;G, H)$ to denote the submatrix of $M$ whose rows and columns are corresponding
to the vertices of $H$. The block index of a vertex $v$  in $G$, denoted by $\mathsf{bi}_G(v)$, equals the number of blocks in $G$ that contains the vertex $v$.

\begin{defn}[Composition bag]
Let $G=(V,E)$ be a distance well-defined graph with blocks $G_t=(V_t,E_t)$; $1\leq t\leq b$. Let $D$ be a generalized distance matrix of $G$. For each $1\leq t\leq b$, let $D_t=\mathsf{sub}(D;G,G_t)$ and let $B_t=(D_t, \lambda_t, \alpha_t, \beta_t, \mathcal{L}_t)$ be a $|V_t|$-bag. The composition bag of bags $B_1,B_2,\ldots, B_b$ is $|V|$-bag  $(D,\lambda, \alpha, \beta, \mathcal{L} )$ whose parameters are defined as follows:
\begin{itemize}
\item[] $\lambda= \sum_{t=1}^b \lambda_t,$
\item[] $\alpha(v) = \sum_{t=1}^b \alpha_t(v) - (\mathsf{bi}_G(v)-1)$,
\item[] $\beta(v) = \sum_{t=1}^b \beta_t(v) - (\mathsf{bi}_G(v)-1)$,
\item[] $\mathcal{L} = \sum_{t=1}^{b}\widehat{\mathcal{L}}_t,$
\end{itemize}
where for any vertex $v$, the entry of $\alpha$ corresponding to $v$ is $\alpha(v)$ and the entry of $\beta$ corresponding to $v$ is $\beta(v)$, and for each  $1\leq t\leq b$,  $\widehat{\mathcal{L}}_t$ is an $|V|\times|V|$ matrix the entry of $G$ is $$(\widehat{\mathcal{L}}_t)_{uv}=\begin{cases}
({\mathcal{L}}_t)_{uv} & \textup{ if } u,v\in V_t,\\
0 & \textup{ otherwise.}
\end{cases}$$
\end{defn}
The next result gives the inverse of a generalized distance matrix of distance well-defined graphs and as a consequence the result is valid for strongly connected digraphs. 
\begin{theorem}\label{thm:extn}\cite[Theorem 3.8]{Zhou2}
Let $G=(V,E)$ be a distance well-defined graph with blocks $G_t=(V_t,E_t)$; $1\leq t\leq b$. Let $D$ be a generalized distance matrix of $G$. For each $1\leq t\leq b$, let $D_t=\mathsf{sub}(D;G,G_t)$ and let $B_t=(D_t, \lambda_t, \alpha_t, \beta_t, \mathcal{L}_t)$ be a $|V_t|$-bag. Let  $(D,\lambda, \alpha, \beta, \mathcal{L} )$ be the  composition bag of bags $B_1,B_2,\ldots, B_b$. Then, $D$ is a left LapExp($\lambda, \alpha, \beta, \mathcal{L}$) matrix (or right LapExp($\lambda, \alpha, \beta, \mathcal{L}$) matrix). Furthermore, if $\lambda\neq 0$, then  $$D^{-1} = -\mathcal{L} + \dfrac{1}{\lambda}\  \beta \alpha^T.$$ 
\end{theorem}

We conclude this section with the result to determine the determinant using the Schur complement technique. 
Let $B$ be an $n\times n$ matrix partitioned as
{\small \begin{equation}\label{eqn:B}
B= \left[
\begin{array}{c|c}
B_{11}& B_{12} \\
\midrule
B_{21} &B_{22}
\end{array}
\right],
\end{equation}}
where $ B_{11}$ and $B_{22}$ are square matrices. If $B_{11}$ is nonsingular, then the Schur complement of $B_{22}$ in $B$ is defined to be the matrix $B/B_{11}=B_{22} - B_{21}B_{11}^{-1}B_{12}$. Similarly, if  $B_{22}$ is nonsingular, then the Schur complement of $B_{22}$ in $B$ is defined to be the matrix $B/B_{22}=B_{11} - B_{12}B_{22}^{-1}B_{21}$.

\begin{prop}\label{prop:blockdet}\cite{Zhang1}
Let $B$ be an $n\times n$ matrix partitioned as in Eqn.~(\ref{eqn:B}). If $B_{11}$ is nonsingular, then  $\det B = \det B_{11} \times \det(B_{22} - B_{21}B_{11}^{-1}B_{12}).$ Similarly, if  $B_{22}$ is nonsingular, then
$\det B = \det B_{22} \times \det(B_{11} - B_{12}B_{22}^{-1}B_{21}).$
\end{prop}

%In the next section, we compute the determinant of the distance matrix for weighted cactoid-type digraphs and find its inverse whenever it exists.

In the next section, we  introduce few conventions to assign the weights on the edges of  the digraph $ dC(n;m_1,\cdots,m_r)$, and compute the determinant and cofactor of the distance matrix for weighted digraph $ dC(n;m_1,\cdots,m_r)^W$.

 \section{Distance Matrix of  $ dC(n;m_1,\cdots,m_r)^W$}\label{sec:Oriented Block}
 
 We begin this section, by recalling the definition of $dC(n;m_1,\cdots,m_r)$ and the reasoning behind the choice of the notation for the blocks of cactoid-type digraphs. Let $G$ be a strongly connected digraph consisting of finitely many oriented cycles sharing a common path $P^{(c)}(say)$ such that the intersection between any two cycles is also precisely the path $P^{(c)}$. For notational convenience, we give the following representation to the digraph $G$. 

Let $P^{(c)}$ be a directed path $u_0 \to u_1 \to \cdots \to u_n$ and for $j = 1,2, \ldots, r$,  let $P^{(j)}$ be a directed path $v_0^{(j)} \to v_1^{(j)} \to \cdots \to v_{m_j}^{(j)} \to v_{m_j+1}^{(j)}$. For all $j = 1,2, \ldots , r$, we identify the vertex $u_0$ with $v_{m_j+1}^{(j)}$, denote it as $u_0$ and similarly for all $j = 1,2, \ldots , r$, we identify the vertex $u_n$ with $v_0^{(j)}$ and denote it as $u_n$. Then the resulting graph $G=(V,E)$ with $|V|=(n+1)+\sum_{j=1}^{r} m_j$ vertices, is a strongly connected digraph consisting of $r$ oriented cycles sharing the common path $P^{(c)}$ as defined above  and we denote such digraphs as $dC(n;m_1,\cdots,m_r)$. Note that whenever  $r=1$, the  digraph $dC(n;m_1)$ is an oriented cycle on $(n+1)+m_1$ vertices, {\it{i.e.},} $dC(n;m_1)=dC_{(n+1)+m_1}$. In this above choice if  we consider  undirected $r$ cycles,  we denote the resulting graph as $C(n;m_1,\cdots,m_r)$.

We  will use the following convention to index the weights on the directed edges. If $e : u \to v$  represents a directed edge from $u$ to $v$,  then we index the weight ``$W(e)$" on the edge $e$ as $W_{v}$.  Now with the identification $u_0=v_{m_j+1}^{(j)}$ and $u_n=v_0^{(j)}$, for all $j=1,2,\cdots , r$, the weights on the edges (see Figure~\ref{fig:dC1}) of $dC(n;m_1,\cdots,m_r)^W$ are denoted as follows: 
\begin{equation}\label{eqn:weights}
W(e)=\begin{cases}
W_{i+1}  &  \text{if} \  e: u_i \to u_{i+1}; \ i = 1,2,\ldots,n-1,\\
W_i^{(j)}  &  \text{if} \ e: v_{i-1}^{(j)} \to v_{i}^{(j)}; \ j = 1,2,\ldots, r,  \ i = 1,2,\ldots,m_j, \\
W_0^{(j)}   &  \text{if} \ e: v_{m_j}^{(j)} \to u_0; \ j = 1,2,\ldots, r.
\end{cases}
\end{equation}
\begin{figure}[ht]
\centering
\begin{tikzpicture}[scale=0.7]
\Vertex[x=0,y=0,label=$u_0$,position=left,size=.2,color=black]{1}
\Vertex[x=0,y=2,label=$u_1$,position=left,size=.2,color=black]{2}
%%\Vertex[x=0,y=8,label=$u_{n-2}$,position=left,size=.2,color=black]{3}
\Vertex[x=0,y=10,label=$u_{n-1}$,position=left,size=.2,color=black]{4}
\Vertex[x=0,y=12,label=$u_n$,position=left,size=.2,color=black]{5}
\Vertex[x=2,y=10,label=$v_1^{(1)}$,position=right,size=.2,color=black]{6}
\Vertex[x=2,y=8,label=$v_2^{(1)}$,position=right,size=.2,color=black]{7}
\Vertex[x=2,y=2,label=$v_{m_1}^{(1)}$,position=right,size=.2,color=black]{8}
\Vertex[x=4,y=10,label=$v_1^{(2)}$,position=right,size=.2,color=black]{9}
\Vertex[x=4,y=8,label=$v_2^{(2)}$,position=right,size=.2,color=black]{10}
\Vertex[x=4,y=2,label=$v_{m_2}^{(2)}$,position=right,size=.2,color=black]{11}
\Vertex[x=9,y=10,label=$v_1^{(r)}$,position=right,size=.2,color=black]{12}
\Vertex[x=9,y=8,label=$v_2^{(r)}$,position=right,size=.2,color=black]{13}
\Vertex[x=9,y=2,label=$v_{m_r}^{(r)}$,position=right,size=.2,color=black]{14}

\Vertex[x=9,y=4,label=$v_{m_r-1}^{(r)}$,position=right,size=.2,color=black]{15}
\Vertex[x=2,y=4,label=$v_{m_1-1}^{(1)}$,position=right,size=.2,color=black]{16}
\Vertex[x=4,y=4,label=$v_{m_2-1}^{(2)}$,position=right,size=.2,color=black]{17}

\Edge[label=$W_1$,position={left=1mm},Direct](1)(2)
\Edge[position={above},style={dashed},Direct](2)(4)
%%\Edge[label=$W_{n-1}$,position={left=1mm},Direct](3)(4)
\Edge[label=$W_n$,position=left,Direct](4)(5)
\Edge[label=$W_1^{(1)}$,position={left=-4mm},Direct](5)(6)
\Edge[label=$W_2^{(1)}$,position=right,Direct](6)(7)
\Edge[style={dashed},Direct](7)(16)
\Edge[label=$W_0^{(1)}$,position={left=-4mm},Direct](8)(1)
\Edge[label=$W_1^{(2)}$,position={right=-4mm},Direct](5)(9)
\Edge[label=$W_2^{(2)}$,position=right,Direct](9)(10)
\Edge[style={dashed},Direct](10)(17)
\Edge[label=$W_0^{(2)}$,position={left=-4mm},Direct](11)(1)
\Edge[label=$W_1^{(r)}$,position={right=1mm},Direct](5)(12)
\Edge[label=$W_2^{(r)}$,position=right,Direct](12)(13)
\Edge[style={dashed},Direct](13)(15)
\Edge[label=$W_0^{(r)}$,position={right=-.5mm},Direct](14)(1)

\Edge[label=$W_{m_r}^{(r)}$,position=right,Direct](15)(14)
\Edge[label=$W_{m_1}^{(1)}$,position=right,Direct](16)(8)
\Edge[label=$W_{m_2}^{(2)}$,position=right,Direct](17)(11)

\draw [dots]  (5,6) -- (8.5,6); 
\end{tikzpicture}
\caption{$dC(n;m_1,\cdots,m_r)^W$}\label{fig:dC1}

\end{figure}
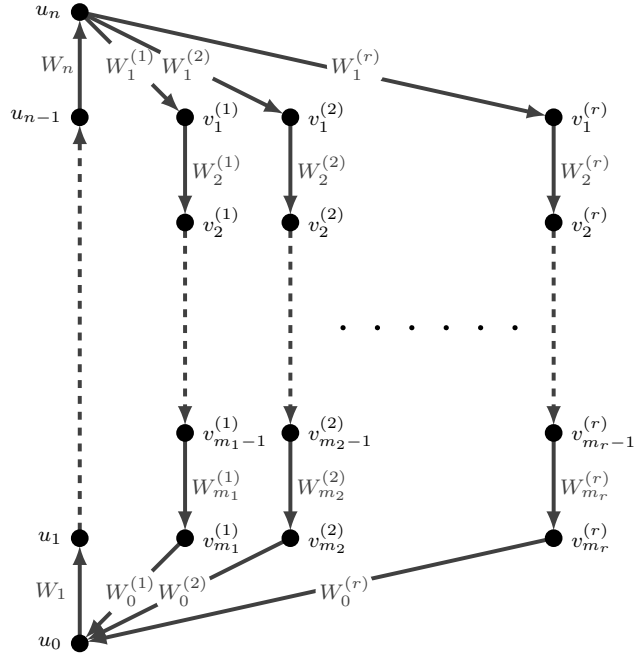\\

Before proceeding further, we first introduce a few notations useful for the subsequent results. For $j=1,\ldots ,r$, let us denote 
\begin{equation}\label{eqn:cycle-weights}
\begin{cases}
w_j= w_c + \widehat{w}_j, \mbox{ where  }\\
\ds w_c=\sum_{i=1}^n W_i \mbox{ and } \widehat{w}_j=\sum_{i=0}^{m_j}W_i^{(j)}.
\end{cases}
\end{equation}
Notice that,  $dC(n;m_1,\cdots,m_r)^W$ consists of $r$ cycles  with total weights $w_1, w_2,\ldots, w_r$. For $1\leq j\leq r$,  let $dC^{(j)}$ denote the oriented cycle due to the path $P^{(c)}$ and $P^{(j)}$ with total weight $w_j$. Without loss of generality, we assume  $w_1\leq w_2\leq \cdots \leq w_r$, this implies  the oriented cycle $dC^{(1)}$ is of least total weight, among all $r$ cycles. Hence $d(u_0,u_n)=w_c$ and $d(u_n,u_0)=\widehat{w}_1$.

 For $1\leq j\leq r,$ let $w_c^{(2)}$ and $w_j^{(2)}$ denote the sum of weights on edges taken two at a time (without repetition) on paths $P^{(c)}$ and $P^{(j)}$, respectively. Then, 
\begin{equation}\label{eqn:w-2}
\begin{cases}
\ds w_c^{(2)}= \sum_{1\leq s<t \leq n}W_s W_t, \\
\\
\ds w_j^{(2)}=\sum_{0\leq s<t \leq m_j}W_s^{(j)} W_t^{(j)}; \ 1\leq j\leq r.
\end{cases}
\end{equation}

\begin{figure}[ht]
     \centering
     \begin{subfigure}{.5\textwidth}
         \centering
\begin{tikzpicture}[scale=0.5]
\Vertex[x=-3,y=1,label=1,position=left,size=.2,color=black]{1}
\Vertex[x=-1,y=3,label=2,position=above,size=.2,color=black]{2}
\Vertex[x=2,y=3,label=3,position=above,size=.2,color=black]{3}
\Vertex[x=4,y=1,position=right,size=.2,color=black]{4}
\Vertex[x=4,y=-1,position=right,size=.2,color=black]{5}
\Vertex[x=2,y=-3,position=below,size=.2,color=black]{6}
\Vertex[x=-1,y=-3,label=n-1,position=below,size=.2,color=black]{n-1}
\Vertex[x=-3,y=-1,label=0,position=left,size=.2,color=black]{n}

\Edge[label=$\theta_1$,position=left,Direct](n)(1)
\Edge[label=$\theta_2$,position={left=1mm},Direct](1)(2)
\Edge[label=$\theta_3$,position={above},Direct](2)(3)
\Edge[position=below,style={dashed},Direct](3)(4)
\Edge[position=below,style={dashed},Direct](4)(5)
\Edge[position=below,style={dashed},Direct](5)(6)
\Edge[position=below,style={dashed},Direct](6)(n-1)
\Edge[label=$\theta_0$,position={left=1mm},Direct](n-1)(n)
 
\end{tikzpicture}
         \caption{ $dC_n^W$}
     \end{subfigure}% <- nötig
    \begin{subfigure}{.5\textwidth}
\centering
\begin{tikzpicture}[scale=0.9]
\Vertex[x=0,y=0,label=$0$,position=left,size=.2,color=black]{A}
\Vertex[x=1,y=2,label=$1$,position=left,size=.2,color=black]{B}
\Vertex[x=3,y=2,label=$2$,position=right,size=.2,color=black]{C}
\Vertex[x=4,y=0,label=$3$,position=right,size=.2,color=black]{D}
\Vertex[x=2,y=-1.5,label=$4$,position=below,size=.2,color=black]{E}

\Edge[label=$\theta_1$,position={left=1mm},Direct](A)(B)
\Edge[label=$\theta_2$,position=above,Direct](B)(C)
\Edge[label=$\theta_3$,position={right=1mm},Direct](C)(D)
\Edge[label=$\theta_4$,position={below=1mm},Direct](D)(E)
\Edge[label=$\theta_0$,position={below=1mm},Direct](E)(A)
\end{tikzpicture}
\caption{$dC_5^W$}%%\label{fig:Fig_1}
\end{subfigure}
     \caption{}\label{fig:1}
\end{figure}

\begin{obs}\label{rem:w-2-sum}
Let $\theta_0,\theta_1,\ldots,\theta_{n-1}$ be  $n$ weights. Let $ w^{(2)}= \sum_{0\leq s<t \leq n-1}\theta_s \theta_t $  be the sum of these weights  taken two at a time (without repetition). Now we will look into different rearrangements to obtain  the sum $w^{(2)}$ which are useful in our subsequent calculations. 

Let us consider the oriented cycle $dC_n$ with vertex set $\{0,1,\ldots, n-1\}$  and let $\theta_i$ be the weight on the edge $i-1\rightarrow i$ (see Figure~\ref{fig:1}(a)).  Observe that, the sum $ w^{(2)}$ can be oriented  by starting with any weight $\theta_s$, when $0\leq s\leq n-1$,  considering the sum in clockwise direction  in $dC_n$  and  using the notation $[i]\equiv i (mod\ n)$, we have
$$ w^{(2)}= \sum_{i=s}^{s+n-2}\left(\theta_{[i]}  \sum_{j=i+1}^{s+n-1}\theta_{[j]}\right)= \sum_{i=s}^{s+n-2}\theta_{[i]} \ d([i], [s-1]).$$
Similarly,  considering the sum in anticlockwise direction  in $dC_n$, we have 
$$w^{(2)}= \sum_{i=0}^{n-2}\left(\theta_{[s-i]}  \sum_{j=s+1}^{s-i-1}\theta_{[j]}\right)= \sum_{i=0}^{n-2}\theta_{[s-i]} \ d(s, [s-i-1]).$$
\noindent{\bf{For example:} }let us consider the sum  $w^{(2)}$, for the weights $\theta_0,\theta_1,\theta_2,\theta_3,\theta_4$ (see Figure~\ref{fig:1}(b)). Now starting with $\theta_2$, the sum in clockwise direction is:
$$w^{(2)}=\theta_2(\theta_3+\theta_4+\theta_0+\theta_1)+\theta_3(\theta_4+\theta_0+\theta_1)+\theta_4(\theta_0+\theta_1)+\theta_0\theta_1,$$
and the sum in anticlockwise direction is:
$$w^{(2)}=\theta_2(\theta_3+\theta_4+\theta_0+\theta_1)+\theta_1(\theta_3+\theta_4+\theta_0)+\theta_0(\theta_3+\theta_4)+\theta_4\theta_3.$$
\end{obs}
The lemmas given  below give  rearrangements of the $w^{(2)}$-sum discussed above  for the weights on  oriented cycle  $dC^{(j)}$ (due to the path $P^{(c)}$ and $P^{(j)}$ in $ dC(n;m_1,\cdots,m_r)^W$) and  on the path $P^{(j)}$, respectively.
%%
%%The lemma below gives a rearrangement of the sum of  weights on  oriented cycle  $dC^{(j)}$ (due to the path $P^{(c)}$ and $P^{(j)}$ in $ dC(n;m_1,\cdots,m_r)$) taken two at a time without repetition.
\begin{lem}\label{lem:sum-over-cycles}
For $1\leq j\leq r$, let $w^{(2)}(dC^{(j)})$ denote the sum of  weights on  $dC^{(j)}$ taken two at a time without repetition. Then 
$$w^{(2)}(dC^{(j)})=w_c \widehat{w}_j+ w_c^{(2)}+w_j^{(2)}.$$
\end{lem}
\begin{proof}
Using Observation~\ref{rem:w-2-sum}, computing the sum in clockwise direction   starting with the weight $W_1$ yields 
{\small \begin{align*}
w^{(2)}(dC^{(j)})&= \sum_{i=1}^{n-1}W_i \left( \sum_{s=i+1}^{n} W_{s}
                     + \sum_{s=0}^{m_j}  W_{s}^{(j)}\right)+ W_n  \sum_{s=0}^{m_j} W_{s}^{(j)}
                     +\sum_{i=1}^{m_j-1}W_{i}^{(j)} \left( W_{0}^{(j)}
                     + \sum_{s=i+1}^{m_j} W_{s}^{(j)}\right) + W_{m_j}^{(j)} W_{0}^{(j)} \\
                 &= \sum_{i=1}^{n-1}W_i \left( \sum_{s=i+1}^{n} W_{s} \right)
                     +\left( W_n+ \sum_{i=1}^{n-1}W_i \right)
                     \left( \sum_{s=0}^{m_j}  W_{s}^{(j)}\right)
                     + \sum_{i=0}^{m_j-1}W_{i}^{(j)} \left(\sum_{s=i+1}^{m_j} W_{s}^{(j)}\right)
\end{align*}}
and the result follows.
\end{proof}

\begin{lem}\label{lem:sum-over-paths}
For $1\leq j\leq r$,  let  $v_s^{(j)}$ be a  vertex on weighted digraph $ dC(n;m_1,\cdots,m_r)^W$. Then 
$$w_j^{(2)}= w_{j}\ d(u_n,v_s^{(j)}) + \left[\left(\sum_{i=1}^{m_{j}}W_{i}^{(j)}d(v_{i}^{(j)},v_s^{(j)})\right)-  d(u_n,v_{m_j}^{(j)}) d(u_0,v_s^{(j)})\right].
$$  
\end{lem}
\begin{proof}
In view of $d(u_n,v_{m_j}^{(j)})=\sum_{i=1}^{m_j}W_i^{(j)}$ and $d(v_s^{(j)},v_s^{(j)})=0$, we have
\begin{align*}
&\left(\sum_{i=1}^{m_{j}}W_{i}^{(j)}d(v_{i}^{(j)},v_s^{(j)})\right)-  d(u_n,v_{m_j}^{(j)}) d(u_0,v_s^{(j)})\\
=&\sum_{i=1}^{m_j}W_{i}^{(j)}  d(v_{i}^{(j)},v_s^{(j)}) - d(u_0,v_s^{(j)}) \sum_{i=1}^{m_j}W_{i}^{(j)} \\
=& \sum_{i=1}^{s-1}W_{i}^{(j)} \left(  d(v_{i}^{(j)},u_0)-d(v_{s}^{(j)},u_0) \right) + \sum_{i=s+1}^{m_j}W_{i}^{(j)} \left(d(v_{i}^{(j)},u_0) +d(u_0,v_s^{(j)})\right) - \sum_{i=1}^{m_j}W_{i}^{(j)} d(u_0,v_s^{(j)})\\
=& \sum_{i=1}^{m_j}W_{i}^{(j)}   d(v_{i}^{(j)},u_0)  - \sum_{i=1}^{s} W_{i}^{(j)}\left(d(v_s^{(j)}, u_0) +d(u_0, v_s^{(j)})\right)\\
=& w_j^{(2)} - w_j \sum_{i=1}^{s} W_{i}^{(j)}= w_j^{(2)} - w_j \ d(u_n,v_s^{(j)})
\end{align*} 
and hence the result follows.
\end{proof}

\subsection{Determinant and Cofactor of $D(dC(n;m_1,\cdots,m_r)^W)$}
In this section, we will calculate the determinant  and  cofactor of the distance matrix  for weighted digraph $ dC(n;m_1,\cdots,m_r)^W$. Using the conventions on weights assigned to edges of $ dC(n;m_1,\cdots,m_r)$ as in Eqns.~\eqref{eqn:weights} and~\eqref{eqn:cycle-weights} with the assumption $w_1\leq w_2\leq \cdots \leq w_r$, the distance matrix of $dC(n;m_1,\cdots,m_r)^W$ is $D(dC(n;m_1,\cdots,m_r)^W)=[d_{xy}]$, where
\begin{equation*}
\ds d_{xy}=
\begin{cases}
0   & \text{if}\ x=y,\\
\ds\sum_{i=p+1}^{q} W_i   & \text{if}\ x=u_{p}, y=u_{q}; \; 0\leq p < q \leq n,\\

\ds w_1 - \sum_{i=p+1}^{q} W_i   & \text{if}\ x=u_{q}, y=u_{p}; \; 0\leq p < q \leq n,\\

\ds  \left( w_c- \sum_{i=1}^p W_i \right) + \sum_{i=1}^q W_i^{(j)}   & \text{if}\ x=u_{p}, y=v_{q}^{(j)}; \; 0\leq p \leq n, 1\leq q \leq m_j, 1\leq j\leq r,  \\

\ds\left(  \widehat{w}_j-  \sum_{i=1}^q W_i^{(j)} \right) +\sum_{i=1}^p W_i     & \text{if}\ x=v_{q}^{(j)}, y=u_{p}; \; 0\leq p \leq n, 1\leq q \leq m_j, 1\leq j\leq r, \\

\ds \sum_{i=p+1}^q W_i^{(j)}   & \text{if}\ x=v_{p}^{(j)},  y=v_{q}^{(j)}; \; 1\leq p < q\leq m_j,  1\leq j\leq r,\\

\ds w_j-\sum_{k=p+1}^q W_k^{(j)}   & \text{if}\ x=v_{q}^{(j)},  y=v_{p}^{(j)}; \; 1\leq p < q\leq m_j,  1\leq j\leq r,\\

\ds w_{j}-\sum_{i=1}^pW_i^{(j)} + \sum_{i=1}^q W_i^{(l)}   & \text{if}\ x=v_{p}^{(j)}, y=v_{q}^{(l)}; \;   1\leq p \leq m_{j}, 1\leq q\leq m_{l},  1\leq j, l\leq r.\\
\end{cases}
\end{equation*}

In the example given below we consider the weighted digraph $dC(2;2,2)^W$, with weights  as shown in  Figure~\ref{figure:Det_zero}, and compute its distance matrix $D(dC(2;2,2)^W)$.\\

\begin{minipage}[ht]{0.35\linewidth}
\centering
\begin{tikzpicture}[scale=1.0]
\Vertex[x=-1,y=0,label=$1$,position=left,size=.2,color=black]{1}
\Vertex[x=-1,y=2,label=$2$,position=left,size=.2,color=black]{2}
\Vertex[x=-1,y=4,label=$3$,position=left,size=.2,color=black]{3}
\Vertex[x=0.5,y=3,label=$4$,position=right,size=.2,color=black]{4}
\Vertex[x=0.5,y=1,label=$5$,position=right,size=.2,color=black]{5}
\Vertex[x=2,y=3,label=$6$,position=right,size=.2,color=black]{6}
\Vertex[x=2,y=1,label=$7$,position=right,size=.2,color=black]{7}

\Edge[label=$2$,position={left=1mm},Direct,fontcolor=red](1)(2)
\Edge[label=$1$,position={left=1mm},Direct,fontcolor=red](2)(3)
\Edge[label=$-1$,position={left=2mm},Direct,fontcolor=red](3)(4)
\Edge[label=$-1$,position={left=1mm},Direct,fontcolor=red](4)(5)
\Edge[label=$-1$,position={left=2mm},Direct,fontcolor=red](5)(1)
\Edge[label=$1$,position={below=2mm},Direct,fontcolor=red](7)(1)
\Edge[label=$2$,position={above=1mm},Direct,fontcolor=red](3)(6)
\Edge[label=$1$,position={right=1mm},Direct,fontcolor=red](6)(7)
\end{tikzpicture}

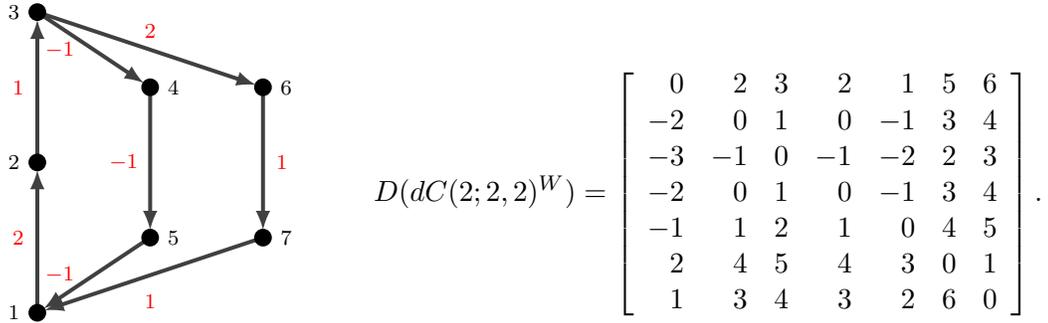
\captionof{figure}{$dC(2;2,2)^W$}\label{figure:Det_zero}
\end{minipage}
\begin{minipage}[ht]{0.63\linewidth}
$
D(dC(2;2,2)^W) = \left[\begin{array}{rrrrrrr}
0 & 2 & 3 & 2 & 1 & 5 & 6 \\
-2 & 0 & 1 & 0 & -1 & 3 & 4 \\
-3 & -1 & 0 & -1 & -2 & 2 & 3 \\
-2 & 0 & 1 & 0 & -1 & 3 & 4 \\
-1 & 1 & 2 & 1 & 0 & 4 & 5 \\
2 & 4 & 5 & 4 & 3 & 0 & 1 \\
1 & 3 & 4 & 3 & 2 & 6 & 0
\end{array}\right].
$
\end{minipage}

\vspace*{.2cm}

 Let $V$ be the vertex set of  $ dC(n;m_1,\cdots,m_r)^W$.  Henceforth, we will use  
$$V=\left\{u_0,u_1,\ldots, u_n,v_{1}^{(1)},\ldots, v_{m_1}^{(1)},v_{1}^{(2)},\ldots, v_{m_2}^{(2)},\ldots, v_{1}^{(r)},\ldots, v_{m_r}^{(r)}\right \}$$ as
the vertex ordering for the distance matrix $D$ of  $ dC(n;m_1,\cdots,m_r)^W$. We begin with the following lemma.
\begin{lem}\label{lem:det-w-j=0}
Let $D$ be the  distance matrix of a weighted  digraph $ dC(n;m_1,\cdots,m_r)^W$. If the total weight $w_j$ of the oriented cycle $dC^{(j)}$ is zero, for some $1\leq j \leq r$, then $det D = 0$.
\end{lem}
\begin{proof}
Assume that the oriented cycle $dC^{(j)}$ in $ dC(n;m_1,\cdots,m_r)^W$ has total weight   $w_j = 0$. Let $x,y,z$ be three vertices in $dC^{(j)}$ such that $x\rightarrow y\rightarrow z$. Let $C_u$ denote the column  corresponding to the vertex $u$ in the distance matrix $D$  and let   $\widetilde{D}$ be the resulting matrix after elementary column operations $C_z \leftarrow C_z - C_y$  followed by $C_y\leftarrow C_y - C_x$ on  $D$. For any vertex $v$ in $ dC(n;m_1,\cdots,m_r)^W$, we have 
\begin{equation*}
(C_z - C_y)(v) = d(v,z) - d(v,y) = 
\begin{cases}
d(y,z) & \text{if } v \neq z,\\
- d(z,y) & \text{if } v = z.\\
\end{cases}
\end{equation*}
Note that $d(z,y)=w_j-d(y,z)$. Since $w_j=0,$ so $(C_z - C_y)(v)= W_z$. Thus, the column corresponding to $z$ in $\widetilde{D}$, {\it{i.e.},} $C_z - C_y= W_z \mathds{1}$ and similar argument  yields that  the column corresponding to $y$ in $\widetilde{D}$,   $C_y - C_x = W_y \mathds{1}$. Hence $\det D = \det \widetilde{D} = 0$.
\end{proof}
The next result  gives the  determinant of the distance matrix of  $dC(n;m_1,\cdots,m_r)^W$.
\begin{theorem}\label{Thm:det_D}
Let $D$ be the  distance matrix of the weighted  digraph $dC(n;m_1,\cdots,m_r)^W$ with $|V|=(n+1)+\sum_{j=1}^{r} m_j$ vertices.  Then, 
$$\det D = 
\begin{cases}
\ds(-1)^{|V|-1} \left( w_1^n \prod_{j=1}^r w_j^{m_j}\right) \left(\frac{w_c \widehat{w}_1}{w_1} + \frac{w_c^{(2)}}{w_1} +  \sum_{j=1}^r \frac{w_j^{(2)}}{w_j} \right) & \textup{ if }  w_j\neq 0;\ 1\leq j\leq r,\\
0 & \textup{ otherwise.}\\
\end{cases}$$
\end{theorem}

\begin{proof}
Let $R_{u_i} / C_{u_i}$ denote the row/column of $D$ corresponding to the vertex $u_i$; $0\leq i\leq n$ and for $1\leq j\leq r$, let $R_{v_i}^{(j)}/ C_{v_i}^{(j)}$ denote the row/column of $D$ corresponding to the vertex $v_i^{(j)}$; $1\leq i\leq m_j$. Now we begin with few elementary matrix operations on the distance matrix $D$, listed below in the following steps:
\begin{itemize}
\item[1.] For each $j=1,2,\ldots,r$,  first  compute $ C_{v_i}^{(j)}\leftarrow C_{v_i}^{(j)}- C_{v_{i-1}}^{(j)}$ recursively, for $i=m_j, m_j-1,\ldots,2$, followed by $C_{v_{1}}^{(j)} \leftarrow C_{v_{1}}^{(j)} - C_{u_{n}}$.

\item[2.]Next, compute $C_{u_i} \leftarrow C_{u_i} - C_{u_{i-1}}$ recursively, for $i=n, n-1,\ldots,1$.

\item[3.] Further, first compute $R_{u_i}  \leftarrow  R_{u_i} - R_{u_{i+1}}$ recursively, for $i=1, 2,\ldots,n-1$, followed by  $R_{u_n}  \leftarrow  R_{u_n} - R_{v_1}^{(1)}$.

\item[4.]Next, for  $j=1,2,\ldots,r$, first compute $R_{v_i}^{(j)}  \leftarrow  R_{v_i}^{(j)} - R_{v_{i+1}}^{(j)}$ recursively, for $i=1,2,\ldots,  m_j-1$, followed by 
$R_{v_{m_j}}^{(j)}  \leftarrow R_{v_{m_j}}^{(j)} - R_{u_0}.$

\item[5.] Now compute $C_{u_{i+1}}\leftarrow C_{u_{i+1}} +  C_{u_{i}}$ recursively 
for $i=1, 2,\ldots,n,$ followed by computing  $C_{v_1}^{(1)} \leftarrow C_{v_1}^{(1)} + C_{u_n} $. %%%Next, for each  $j=1,2,\ldots,r$, compute $C_{v_i}^{(j)} \leftarrow  C_{v_i}^{(j)} - C_{v_{i-1}}^{(j)} $ recursively, for $i = 2,3,\cdots,m_j$. 

\item[6.] Finally,  for each  $j = 2,3, \ldots, r$, compute $C_{v_i}^{(j)} \leftarrow  C_{v_i}^{(j)} + C_{v_{i-1}}^{(j)} $ recursively, for $i = 2,3,\ldots,m_j$.       
\end{itemize} 
Then, the resulting matrix is of the following block form 
$\left[ \begin{array}{c|c}
0 & X^T\\
\hline
Y & Z
\end{array}\right]$, where 
{\small$$Z = \left[ \begin{array}{c|c|c|c}
-w_1 I_{n+m_1} & \mathbf{0} & \cdots &\mathbf{0}\\
\hline
\mathbf{0} & -w_2 I_{m_2} & \cdots &\mathbf{0}\\
\hline
\vdots & \vdots & \ddots &\vdots\\
\hline
\mathbf{0}&\mathbf{0}& \cdots& -w_r I_{m_r}\\
\end{array}\right],$$
$$X^T=[X_1^T|X_2^T|\cdots |X_r^T] \mbox{ and } Y^T=[Y_1^T|Y_2^T|\cdots |Y_r^T] \mbox{ with}$$ 
\begin{equation}\label{eqn:1}
\begin{cases}
\ds X_1^T=\left[W_1,\sum_{i=1}^{2}W_i, \cdots, \sum_{i=1}^{n}W_i,  \sum_{i=1}^{n}W_i + W^{(1)}_1,\cdots,   \sum_{i=1}^{n}W_i  +\sum_{i=1}^{m_1}W^{(1)}_i\right], \\
\\  
Y_1^T=\left[W_2,W_3,\cdots, W_n,W^{(1)}_1, \cdots, W^{(1)}_0 \right],\\
\\  
\ds X_j^T=\left[W^{(j)}_1 ,\sum_{i=1}^{2}W^{(j)}_i , \cdots, \sum_{i=1}^{m_j}W^{(j)}_i \right], \ \  Y_j^T=\left[W^{(j)}_2, W^{(j)}_3, \cdots, W^{(j)}_0\right];\   2\leq j\leq r.
\end{cases}
\end{equation}}

If $w_j=0$ for some $j,$ then result follows from Lemma~\ref{lem:det-w-j=0}. Next, assume that  $w_j\neq 0;\ 1\leq j\leq r$. Then $Z$ is invertible, so using Proposition~\ref{prop:blockdet}, we have 
\begin{equation}\label{eqn:det1}
\det D= (\det Z)\times (\det(-X^TZ^{-1}Y)).
\end{equation}
Further, using Eqn.~\eqref{eqn:1}, we have 
\begin{align}\label{eqn:det2}
\det(-X^TZ^{-1}Y)=\sum_{j=1}^{r} \frac{\langle X_j,Y_j \rangle}{w_j}= \frac{\langle X_1,Y_1 \rangle}{w_1}+ \sum_{j=2}^{r} \frac{\langle X_j,Y_j \rangle}{w_j},
\end{align} 
where $\langle X_j,Y_j \rangle$ represent the inner product of the column vectors $ X_j $ and $Y_j$. By Observation~\ref{rem:w-2-sum},   starting the sum with $W_{0}^{(1)}$, in anticlockwise direction with the weights on $dC^{(1)}$ and using Lemma~\ref{lem:sum-over-cycles}, we have
{\small 
\begin{align*}
\noindent\langle X_1,Y_1 \rangle &= W_2W_1+ W_3\left(\sum_{i=1}^{2}W_i\right)
            +\cdots+ W_n \left(\sum_{i=1}^{n-1}W_i\right)
+ W_1^{(1)}\left(\sum_{i=1}^{n}W_i\right) + W_2^{(1)}\left(W_1^{(1)}+\sum_{i=1}^{n}W_i\right)\\\nonumber
     &\hspace*{2cm}+\cdots + W_{m_1}^{(1)}\left(\sum_{i=1}^{m_1-1}W_i^{(1)}+\sum_{i=1}^{n}W_i\right)+ W_{0}^{(1)}\left(\sum_{i=1}^{m_1}W_i^{(1)}+\sum_{i=1}^{n}W_i\right)\\ \nonumber
     &= w^{(2)}(dC^{(1)})=w_c \widehat{w}_1+ w_c^{(2)}+w_1^{(2)}. 
\end{align*}}
Similarly,  for $1\leq j\leq r$, considering the sum starting with $W_{0}^{(j)}$ in anticlockwise direction with the weights on $P^{(j)}$, we have
{\small \begin{flalign*}
\langle X_j,Y_j \rangle &= W_2^{(j)}W_1 ^{(j)}+ W_3^{(j)}\left(\sum_{i=1}^{2}W_i^{(j)}\right)
            +\cdots+ W_n^{(j)} \left(\sum_{i=1}^{m_j-1}W_i^{(j)}\right)
+ W_0^{(j)}\left(\sum_{i=1}^{m_j} W_i^{(j)}\right) = w_j^{(2)}.      
     \end{flalign*}}
Using  the above calculations in Eqns.\eqref{eqn:det1} and \eqref{eqn:det2}, yields the result.
\end{proof}
Next we compute the cofactor of the  distance matrix of the weighted  digraph $dC(n;m_1,\cdots,m_r)^W$.
\begin{theorem}\label{Thm:cof_D}
Let $D$ be the  distance matrix of the weighted  digraph $dC(n;m_1,\cdots,m_r)^W$  with $|V|=(n+1)+\sum_{j=1}^{r} m_j$ vertices. Then the cofactor of the distance matrix $D$ is given by  $$\cof D = (-1)^{|V|-1} w_1^n \prod_{j=1}^r w_j^{m_j}.$$
\end{theorem}

\begin{proof}
Let $M$ be the matrix obtained from $D$ by subtracting the first row from all other rows and then subtracting the first column from all other columns. Then the matrix $M(1|1)$ in block form is given by
{\small\begin{equation*}
M(1|1) =
\left[
\begin{array}{c|c|c|c|c}
-w_1 U_{n} & -w_1 J_{n\times m_1} &-w_1 J_{n\times m_2}  &\cdots &-w_1 J_{n\times m_r}\\
\hline
\mathbf{0} & - w_1 U_{m_1} & \mathbf{0} &\cdots &\mathbf{0}  \\
\hline
\mathbf{0} &\mathbf{0} &  -w_2 U_{m_2} &\cdots &\mathbf{0}  \\
\hline
 &\vdots & \vdots & \ddots & \vdots  \\
\hline
\mathbf{0} & \mathbf{0} & \mathbf{0}  &\cdots & - w_r U_{m_r} \\
\end{array}
\right],
\end{equation*}} where $U_s$ is an upper triangular  matrix of order $s$, with all its non-zero entries $1$. Note that, $M(1|1)$ is obtained after deleting row and column of $M$ corresponding to the vertex $u_0$. Using similar notations for rows and columns  as in Theorem~\ref{Thm:det_D}, we use the following elementary column operations  on $M(1|1)$:
\begin{itemize}
\item[1.]For each $j = 1,2,\ldots,r$, first compute  $C_{v_i}^{(j)} \leftarrow C_{v_i}^{(j)} - C_{v_{i-1}}^{(j)}$, recursively for $i = m_j,m_{j-1}, \ldots, 2$ followed by $C_{v_1}^{(j)} \leftarrow C_{v_1}^{(j)} - C_{u_n}$.
\item[2.] Next, compute $C_{u_i} \leftarrow C_{u_i} - C_{u_{i-1}}$, recursively for $i = n, n-1, \ldots, 2.$
\end{itemize}
The resulting matrix is given by
{\small\begin{equation*}
\left[
\begin{array}{c|c|c|c|c}
-w_1 I_{n} & \mathbf{0} &\mathbf{0}  &\cdots &\mathbf{0}\\
\hline
\mathbf{0} &- w_1 I_{m_1} & \mathbf{0} &\cdots &\mathbf{0}  \\
\hline
\mathbf{0} &\mathbf{0} & - w_2 I_{m_2} &\cdots &\mathbf{0}  \\
\hline
 &\vdots & \vdots & \ddots & \vdots  \\
\hline
\mathbf{0} & \mathbf{0} & \mathbf{0}  &\cdots & -w_r I_{m_r} \\
\end{array}
\right], 
\end{equation*}}
and  hence the result follows from Lemma~\ref{Lem:cof}.
\end{proof}

By Theorem~\ref{Thm:det_D}, the $\det D\neq 0$ iff $w_j\neq 0;\ 1\leq j\leq r$ and  $\frac{w_c \widehat{w}_1}{w_1} + \frac{w_c^{(2)}}{w_1} +  \sum_{j=1}^r \frac{w_j^{(2)}}{w_j}\neq 0.$ Thus, Theorem~\ref{Thm:cof_D} gives,  $\det D\neq 0$ implies that $\cof D\neq 0$. The next section deals with the inverse of the distance matrix $D$ of weighted digraph $dC(n;m_1,\cdots,m_r)^W$, so henceforth we only consider the case whenever  $\det D\neq 0$. It is easy to see for weighted digraph $dC(n;m_1,\cdots,m_r)^W$, if the weight on each edge is positive,  then $\det D\neq 0$.

\section{Inverse of  $ D( dC(n;m_1,\cdots,m_r)^W)$}\label{sec:inv-single}

In this section, we find the inverse  of the  distance matrix  $D$ for weighted  digraph $ dC(n;m_1,\cdots,m_r)^W$,  by proving that $D$ is a left LapExp($\lambda, \alpha, \beta, \mathcal{L}$) matrix for $dC(n;m_1,\cdots,m_r)^W$. We begin with defining the parameters $\lambda, \alpha,\beta$ and subsequently show that these parameters satisfying the requisite conditions. We define 
{\small\begin{equation}\label{eqn:parameters}
 \begin{cases}
 \vspace*{.2cm}
 \lambda = \dfrac{\det D}{\cof D} = \dfrac{w_c \widehat{w}_1}{w_1} + \dfrac{w_c^{(2)}}{w_1} +  {\ds \sum_{j=1}^r} \dfrac{w_j^{(2)}}{w_j},\\ 
\vspace*{.4cm}
 \alpha^T = \begin{bmatrix}
{\ds \sum_{j=1}^{r}}\dfrac{ W_0^{(j)}}{w_j} - {\ds \sum_{j=2}^{r}} \dfrac{\widehat{w}_j}{w_j}, & \dfrac{W_1}{w_1}, & \cdots, & \dfrac{W_n}{w_1}, & \dfrac{W_1^{(1)}}{w_1}, \cdots \dfrac{W_{m_1}^{(1)}}{w_1}, & \cdots, & \dfrac{W_1^{(r)}}{w_1}, \cdots, \dfrac{W_{m_r}^{(r)}}{w_r}
\end{bmatrix},\\
%\vspace*{.4cm}
\beta^T = \begin{bmatrix}
\dfrac{W_1}{w_1}, & \cdots, & \dfrac{W_n}{w_1}, & {\ds \sum_{j=1}^{r}} \dfrac{W_1^{(j)}}{w_j}- {\ds\sum_{j=2}^{r}} \dfrac{\widehat{w}_j}{w_j}, & \dfrac{W_2^{(1)}}{w_1}, \cdots, \dfrac{W_0^{(1)}}{w_1}, & \cdots, & \dfrac{W_2^{(r)}}{w_r}, \cdots, \dfrac{W_0^{(r)}}{w_r}
\end{bmatrix}.
\end{cases}
\end{equation}}

\begin{lem}\label{lem:alpha}
Let $\alpha, \beta$ be the vectors as defined in Eqn.~\eqref{eqn:parameters}. Then $\alpha^T \mathds{1} = \beta^T \mathds{1}= 1.$
\end{lem}

\begin{proof}
Since $${\ds \sum_{j=1}^{r}}\frac{ W_0^{(j)}}{w_j} - {\ds \sum_{j=2}^{r}} \frac{\widehat{w}_j}{w_j}= \frac{W_0^{(1)}}{w_1}+{\ds \sum_{j=2}^{r}}\left(\frac{ W_0^{(j)} - \widehat{w}_j}{w_j}\right)= \frac{W_0^{(1)}}{w_1} -\sum_{j=2}^r\sum_{i = 1}^{m_j} \frac{W_i^{(j)}}{w_j},$$
so using the definition of $\alpha$, we have
\begin{align*}
\alpha^T \mathds{1} & = \frac{W_0^{(1)}}{w_1} -\sum_{j=2}^r\sum_{i = 1}^{m_j} \frac{W_i^{(j)}}{w_j} + \sum_{i=1}^n \frac{W_i}{w_1}+ \sum_{j=1}^r\sum_{i = 1}^{m_j} \frac{W_i^{(j)}}{w_j}\\
						   & = \frac{W_0^{(1)}}{w_1} + \sum_{i=1}^n \frac{W_i}{w_1}+ \sum_{i = 1}^{m_1} \frac{W_i^{(1)}}{w_1} = 1,
\end{align*}
and similar calculations leads to $\beta^T \mathds{1}= 1$.
\end{proof}

\begin{lem}\label{lem:alpha*D}
Let $D$ be the  distance matrix of the weighted  digraph $dC(n;m_1,\cdots,m_r)^W$. Then $\alpha^T D = \lambda \mathds{1}^T$, where  
$\lambda$ and $\alpha$ are  defined as in Eqn.~\eqref{eqn:parameters}. 
\end{lem}

\begin{proof}
 Let $\eta^T= \alpha^T D$ and let $V$ be the vertex set of the weighted  digraph $dC(n;m_1,\cdots,m_r)^W$. We will  prove this result by repeated application of Observation~\ref{rem:w-2-sum},  to show $$\ds \lambda=\eta(v)=\sum_{z\in V} \alpha(z) d(z,v), \mbox{ for all } v \in V.$$
 \noindent $\underline{\textbf{Case 1:}}$ For $v=u_0$, 
\begin{align*}
\eta(v) & =\sum_{i=1}^n\frac{W_i}{w_1}d(u_i,u_0)+\sum_{j=1}^r \sum_{i=1}^{m_j}\frac{W_{i}^{(j)}}{w_j}d(v_{i}^{(j)},u_0)\\
 & = \frac{1}{w_1}\left[\sum_{i=1}^nW_i \ d(u_i,u_0)+ \sum_{i=1}^{m_1}W_{i}^{(1)} d(v_{i}^{(1)},u_0)\right]+\sum_{j=2}^r  \frac{1}{w_j} \left(\sum_{i=1}^{m_j}W_{i}^{(j)} d(v_{i}^{(j)},u_0)\right)\\
 &= \Sigma_{11} +   \Sigma_{12}.
\end{align*}
By  Observation~\ref{rem:w-2-sum},  $\Sigma_{11}$ corresponds to the clockwise sum of weights on  $dC^{(1)}$ starting with $W_1$. Next, for $2\leq j\leq r$ each term in  the sum $\Sigma_{12}$, corresponds to the clockwise sum of weights on path $P^{(j)}$ starting with $W_1^{(j)}$. Therefore, using  Lemma~\ref{lem:sum-over-cycles}, we have
$$\eta(v)= \frac{ w^{(2)}(dC^{(1)})}{w_1}+\sum_{j=2}^r \frac{ w^{(2)}_j}{w_j}=\lambda.$$ 

\noindent $\underline{\textbf{Case 2:}}$ For $v \in \{u_1,u_2,\ldots,u_n,v_1^{(1)},\ldots,v_{m_1}^{(1)}\}$.\\

\noindent Note that, $ { \ds\sum_{j=1}^{r}}\dfrac{ W_0^{(j)}}{w_j} - {\ds \sum_{j=2}^{r}} \dfrac{\widehat{w}_j}{w_j}=  \ds \frac{W_0^{(1)}}{w_1} - \sum_{j=2}^{k} \frac{d(u_n,v_{m_j}^{(j)})}{w_j}$. Thus

\begin{align*}
\eta(v) & = \left( \frac{W_0^{(1)}}{w_1} - \sum_{j=2}^{r} \frac{d(u_n,v_{m_j}^{(j)})}{w_j} \right) d(u_0,v) + \sum_{i=1}^n\frac{W_i}{w_1}d(u_i,v)+ \sum_{j=1}^r \sum_{i=1}^{m_j}\frac{W_{i}^{(j)}}{w_j}d(v_{i}^{(j)},v)\\
&= \frac{1}{w_1}\left[W_0^{(1)} d(u_0,v) +  \sum_{i=1}^n W_i\ d(u_i,v) +  \sum_{i=1}^{m_1} W_i^{(1)}d(v_i^{(1)},v)\right]\\
& \hspace*{6cm}  +\sum_{j=2}^r \frac{1}{w_j}\left[\left(\sum_{i=1}^{m_j}W_{i}^{(j)}d(v_{i}^{(j)},v)\right)-  d(u_n,v_{m_j}^{(j)}) d(u_0,v)\right]\\
 &= \Sigma_{21} +   \Sigma_{22}.
\end{align*}
For $2 \leq j \leq r$, $1 \leq i \leq m_j$, we have $d(v_{i}^{(j)},v) = d(v_{i}^{(j)},u_0) + d(u_0,v)$. So
\begin{align*}
\Sigma_{22}&=\sum_{j=2}^r \frac{1}{w_j}\left[\sum_{i=1}^{m_j}W_{i}^{(j)}d(v_{i}^{(j)},u_0)+\left(\sum_{i=1}^{m_j}W_{i}^{(j)}\right)d(u_0,v)-  d(u_n,v_{m_j}^{(j)}) d(u_0,v)\right]\\
&=\sum_{j=2}^r \frac{1}{w_j}\left(\sum_{i=1}^{m_j}W_{i}^{(j)}d(v_{i}^{(j)},u_0)\right)=\Sigma_{12}=\sum_{j=2}^r \frac{ w^{(2)}_j}{w_j},
\end{align*}	 
and the sum $\Sigma_{21}$ corresponds to the clockwise sum of weights on  $dC^{(1)}$ starting with $W_z$, where $v\to z$ and $z \in dC^{(1)} $. Hence,  similar argument as in the  Case 1 yields $\eta(v)=\lambda.$\\

\noindent$\underline{\textbf{Case 3:}}$ For $v\in \{ v_{i}^{(j)} : \ 2 \leq j \leq r$, $1 \leq i \leq m_j\}$.\\

\noindent Let $v= v_{s}^{(j_0)}$. Similar to Case 2, using  $ { \ds\sum_{j=1}^{r}}\dfrac{ W_0^{(j)}}{w_j} - {\ds \sum_{j=2}^{r}} \dfrac{\widehat{w}_j}{w_j}=  \ds \frac{W_0^{(1)}}{w_1} - \sum_{j=2}^{k} \frac{d(u_n,v_{m_j}^{(j)})}{w_j}$  and along with the fact that $d(y,v)=d(y,u_n)+d(u_n, v)$ for every  vertex $y$ in $dC^{(1)}$, we have
\begin{align*}
\eta(v) &= \frac{1}{w_1}\left[W_0^{(1)} d(u_0,u_n) +  \sum_{i=1}^n W_i\ d(u_i,u_n) +  \sum_{i=1}^{m_1} W_i^{(1)}d(v_i^{(1)},u_n)\right] \\ %%+ \frac{1}{w_1} \left[  \sum_{i=1}^n W_i +  \sum_{i=0}^{m_1} W_i^{(1)}\right] d(u_n,v)\\
& \hspace*{4cm}  + \frac{1}{w_1} \left[  \sum_{i=1}^n W_i +  \sum_{i=0}^{m_1} W_i^{(1)}\right] d(u_n,v)\\
& \hspace*{6cm}  +\sum_{j=2}^r \frac{1}{w_j}\left[\left(\sum_{i=1}^{m_j}W_{i}^{(j)}d(v_{i}^{(j)},v)\right)-  d(u_n,v_{m_j}^{(j)}) d(u_0,v)\right]\\
&= \frac{1}{w_1}\left[W_0^{(1)} d(u_0,u_n) +  \sum_{i=1}^n W_i\ d(u_i,u_n) +  \sum_{i=1}^{m_1} W_i^{(1)}d(v_i^{(1)},u_n)\right]  \\
& \hspace*{2cm}  + \left[ d(u_n,v) + \frac{1}{w_{j_0}}\left[\left(\sum_{i=1}^{m_{j_0}}W_{i}^{(j_0)}d(v_{i}^{(j_0)},v)\right)-  d(u_n,v_{m_j}^{(j_0)}) d(u_0,v)\right] \right]\\
& \hspace*{6cm}  +\sum_{j=2\atop j\neq j_0}^r \frac{1}{w_j}\left[\left(\sum_{i=1}^{m_j}W_{i}^{(j)}d(v_{i}^{(j)},v)\right)-  d(u_n,v_{m_j}^{(j)}) d(u_0,v)\right]\\
&= \Sigma_{31} +   \Sigma_{32} +\Sigma_{33}.
\end{align*}
Note that, the sum $\Sigma_{31}$ corresponds to the clockwise sum of weights on  $dC^{(1)}$, starting with $W_0^{(1)}$ and for  $\Sigma_{33}$, proceeding similar to Case 2,  whenever $2 \leq j \leq r$; $1 \leq i \leq m_j$,  with $j\neq j_0$, we have $d(v_{i}^{(j)},v) = d(v_{i}^{(j)},u_0) + d(u_0,v)$ and hence 
$$\Sigma_{31} +\Sigma_{33} = \frac{ w^{(2)}(dC^{(1)})}{w_1}+\sum_{j=2\atop j\neq j_0}^r \frac{ w^{(2)}_j}{w_j}.$$
By Lemma~\ref{lem:sum-over-paths}, $\Sigma_{32}= \dfrac{ w^{(2)}_{j_0}}{w_{j_0}}$ and this completes the proof.
\end{proof}
\noindent Consider an adjacency matrix  $A^{(In)}=[A_{xy}]$ of  weighted  digraph $dC(n;m_1,\cdots,m_r)^W$, where 
$$A_{xy}=
\begin{cases}
\frac{1}{w_1}  & \text{if}\  x \to y \mbox{ and } x \in \{u_0,u_1,\dots,u_{n-1}\},\\
\frac{1}{w_j} &\text{if} \  x = u_n \text{ and } y = v_1^{(j)}; 1\leq j\leq r,\\
\frac{1}{w_j}  & \text{if}\  x \to y \mbox{ and } x= v_i^{(j)}; 1\leq i\leq m_j, \ 1\leq j\leq r,\\
0 & \text{ otherwise,}
\end{cases}$$
and let $D^{(Out)}=diag(\partial(x))$ be  the diagonal matrix, where $\partial(x)= \sum_{y}A_{x,y}$, the out degree of $x$. 

Let  $ L= D^{(Out)}- A^{(In)} $ denote the Laplacian matrix of the  weighted digraph  $dC(n;m_1,\cdots,m_r)^W$, which satisfies $L\mathds{1}=0$ and similar matrices for directed graphs were studied in~\cite{Bauer}. Now we define  the Laplacian-like matrix  $\mathcal{L}$ of  weighted  digraph $dC(n;m_1,\cdots,m_r)^W$,  as the Laplacian matrix $L$ perturbed  by a matrix $\widehat{L}$ such that  $\mathcal{L}\mathds{1}= \mathds{1}^T\mathcal{L}=0$ and is given by
\begin{equation}\label{eqn:lap}
 \mathcal{L}= L+ \widehat{L}, 
\end{equation}
with  $\widehat{L}= [\widehat{L}_{xy}]$, where 
\begin{equation*}\label{eqn: pert}
\widehat{L}_{xy}=
\begin{cases}
\sum_{j=2}^r \frac{1}{w_j}  &  \text{if}\  x=u_n \mbox{ and } y=u_0,\\
-\sum_{j=2}^r \frac{1}{w_j}  & \text{if} \ x=y=u_n,\\
0   & \text{ otherwise.}

\end{cases}
\end{equation*}

The lemma below gives one of the requisite condition for the distance matrix $D$ to be a Laplacian-like expressible matrix. 

\begin{lem}\label{lem:LD+I}
Let $D$ be the  distance matrix of the weighted  digraph $dC(n;m_1,\cdots,m_r)^W$ and $\mathcal{L}$ be the Laplacian-like matrix as defined in Eqn.~\eqref{eqn:lap}. Then $\mathcal{L}D + I = \beta \mathds{1}^T$, where $\beta$ is as defined  in Eqn.~\eqref{eqn:parameters}.
\end{lem}
\begin{proof}
Let $\widehat{L}D=[(\widehat{L}D)_{xy}]$ and it is easy to see
\begin{equation}\label{eqn:pert-D}
(\widehat{L}D)_{xy}=
%%\begin{cases}
%% \ds \sum_{j=2}^r \frac{1}{w_j}\left( d(u_0,y)-d(u_n,y)\right) & \text{if } x=u_n,\\
%%0, & otherwise.
%%\end{cases}
%%=
\begin{cases}
\vspace*{.2cm}
\ds - \sum_{j=2}^r \dfrac{\widehat{w}_1}{w_j} & \text{if } x=u_n, \  y=u_0,u_1,\ldots,u_{n-1},\\
 \vspace*{.2cm}
\ds  \sum_{j=2}^r \dfrac{w_c}{w_j}, & \text{if } x=u_n, \  y= u_n, v_i^{(j)};\  1\leq i\leq m_j, \ 1\leq j\leq r,\\ 
0 & \text {otherwise.}
\end{cases}
\end{equation}
Now we will compute $LD+I=[(LD+I)_{xy}]$ and using the definition of Laplacian matrix $ L= D^{(Out)}- A^{(In)} $, we have
\begin{equation}\label{eqn:LD+I}
(LD+I)_{xy}=
\begin{cases}
\vspace*{.2cm}
\ds \sum_{x\to z} \frac{d(x,y)- d(z,y)}{w_z} & \text{if } x\neq y,\\
\ds 1- \sum_{x\to z} \frac{d(z,x)}{w_z} & \text{if } x= y,
\end{cases}
\end{equation}
where $$w_z= \begin{cases}w_1  & \text{if }  z \in \{u_0,u_1,\ldots,u_n\},\\ 
w_j & \text{if } z\in\{ v_i^{(j)}:\  1\leq i\leq m_j, \ 1\leq j\leq r\}.\end{cases}$$
\newpage

\noindent $\underline{\textbf{Case 1:}}$ For $x\neq u_n$.\\

Using  Eqn.~\eqref{eqn:LD+I} for $x\neq u_n$,    we have
$$(LD+I)_{xy}=
\begin{cases}
\vspace*{.1cm}
\ds\frac{W_z}{w_z} & \text{if } x\to z  \text{ and }x\neq y,\\
\ds 1- \frac{w_z- W_z}{w_z} & \text{if } x\to z  \text{ and } x= y,
\end{cases}$$
and hence $\ds (LD+I)_{xy}= \frac{W_z}{w_z}$,  where $x\to z$.

\noindent $\underline{\textbf{Case 2:}}$ For $x= u_n$.\\

For  $x= u_n$ and $y\in\{u_0,u_1,\ldots,u_{n-1}\}$, we get  $d(u_n,y)=\widehat{w}_1 + d(u_0,y)$ and  $d(v_1^{(j)},y)=( \widehat{w}_j- W_1^{(j)})+ d(u_0,y)$; $1\leq j\leq r$. Then,  Eqn.~\eqref{eqn:LD+I} yields
\begin{align*}
(LD+I)_{xy}=   \sum_{j=1}^{r} \frac{d(u_n,y)- d(v_1^{(j)},y)}{w_j}
         =\sum_{j=1}^{r} \frac{\widehat{w}_1- (\widehat{w}_j- W_1^{(j)})}{w_j}
= \left[\sum_{j=1}^{r}\frac{W_1^{(j)}}{w_j}- \sum_{j=2}^{r}\frac{\widehat{w}_j}{w_j}\right] + \sum_{j=2}^{r}\frac{\widehat{w}_1}{w_j}.
\end{align*}

For $x=u_n=y$, using Eqn.~\eqref{eqn:LD+I} and  $d(v_1^{(j)},u_n)=w_j- W_1^{(j)}$; $1\leq j\leq r$, we have 
\begin{align*}
(LD+I)_{xy}&= 1-  \sum_{j=1}^{r} \frac{d(v_1^{(j)},u_n)}{w_j}= 1- \sum_{j=1}^{r}\frac{w_j- W_1^{(j)}}{w_j}= \sum_{j=1}^{r}\frac{W_1^{(j)}}{w_j} -(r-1)\\
         &= \sum_{j=1}^{r}\frac{W_1^{(j)}}{w_j}- \sum_{j=2}^{r}\frac{w_c+\widehat{w}_j}{w_j}
         = \left[\sum_{j=1}^{r}\frac{W_1^{(j)}}{w_j}- \sum_{j=2}^{r}\frac{\widehat{w}_j}{w_j}\right] - \sum_{j=2}^{r}\frac{w_c}{w_j}.
\end{align*}

For $x=u_n$ and $y= v_i^{(j_0)} $; for some $1\leq j_0\leq r$ and $1\leq i\leq m_{j_0}$. Note that,  if $ j\neq j_0$, then $d(v_1^{(j)},v_i^{(j_0)})= (w_j- W_1^{(j)})+ d(u_n, v_i^{(j_0)})$. Again, using Eqn.~\eqref{eqn:LD+I}, we get
\begin{align*}
(LD+I)_{xy}&= \frac{d(u_n,v_i^{(j_0)})- d(v_1^{(j_0)},v_i^{(j_0)})}{w_j} + 
 \sum_{j=1\atop j\neq j_0}^{r} \frac{d(u_n,v_i^{(j_0)})- d(v_1^{(j)},v_i^{(j_0)})}{w_j}  \\  
 &=\frac{ W_1^{(j_0)}}{w_{j_0}}  + \sum_{j=1\atop j\neq j_0}^{r} \frac{W_1^{(j)}-w_j}{w_j}
 = \sum_{j=1}^{r}\frac{W_1^{(j)}}{w_j} -(r-1)\\
 &= \left[\sum_{j=1}^{r}\frac{W_1^{(j)}}{w_j}- \sum_{j=2}^{r}\frac{\widehat{w}_j}{w_j}\right] - \sum_{j=2}^{r}\frac{w_c}{w_j}.
\end{align*}
By Eqn.~\eqref{eqn:lap},  $\mathcal{L}D + I = (LD+I)+\widehat{L}D.$ So using Eqn.~\eqref{eqn:pert-D} and the above cases, we have shown $(\mathcal{L}D + I)_{xy}= \beta(x) $, for every $y$. Hence the result follows.
\end{proof}

The next result gives the inverse of the  distance matrix of $dC(n;m_1,\cdots,m_r)^W$, whenever it exists.

\begin{theorem}\label{thm:rank-one-single}
Let $D$ be the  distance matrix of the weighted  digraph $dC(n;m_1,\cdots,m_r)^W$. Let $\mathcal{L}$ be the Laplacian-like matrix as defined in Eqn.~\eqref{eqn:lap} and $\lambda,\alpha, \beta$ be as defined in Eqn.~\eqref{eqn:parameters}. Then $D$ is a left LapExp($\lambda, \alpha, \beta, \mathcal{L}$) matrix. Moreover, if $\det D \neq 0$, then $$D^{-1} = -\mathcal{L} + \dfrac{1}{\lambda}\  \beta \alpha^T.$$
\end{theorem}
\begin{proof}
Using Lemmas~\ref{lem:alpha}, \ref{lem:alpha*D} and~\ref{lem:LD+I}, it is clear that $D$ is a left LapExp($\lambda, \alpha, \beta, \mathcal{L}$) matrix. Further,  $\det D \neq 0$  implies that  $\lambda \neq 0$ and so by Lemma~\ref{lem:inv}, the result follows.  
\end{proof}
\begin{rem}
Calculations similar to Lemmas~\ref{lem:alpha*D} and~\ref{lem:LD+I}, also leads to $D\beta=\lambda\mathds{1}   \mbox{ and } D\mathcal{L} +I = \mathds{1}\alpha^T$ and hence the distance matrix $D$ of the weighted digraph  $dC(n;m_1,\cdots,m_r)^W$,  is also a  right LapExp($\lambda, \alpha, \beta, \mathcal{L}$) matrix.
\end{rem}

The next section deals with the determinant and inverse of the distance matrix for weighted cactoid-type digraphs.

\section{Distance Matrix of  Weighted Cactoid-type Digraphs:\\ Determinant and Inverse}\label{sec:cac-type}

Let us recall, a strongly connected digraph is said to be cactoid-type, if each of its blocks is a  digraph $dC(n;m_1,\cdots,m_r)$, whenever $n,r\geq 1$ and $m_j\geq 1$; $1\leq j\leq r$.  We begin with the result that gives the determinant of  distance matrix for weighted cactoid-type digraph  which  follows directly from Theorem~\ref{Thm:cof-det}. But in view of  Theorems~\ref{Thm:det_D} and \ref{Thm:cof_D} it is possible to find a exact form for the determinant.
\begin{theorem}\label{lem:det-w-cactoid}
Let $G=(V,E)$ be a cactoid-type digraph with blocks $G_t$; $1\leq t\leq b$ and $W:E\to \mathbb{R}$ be a weight function. For $1\leq t\leq b$, let $D$ and $D_t$ be the   distance matrix of weighted digraphs $G^W$ and $G_t^W$, respectively. If $\cof D_t \neq 0$, for all $1\leq t\leq b$, then the determinant of the distance matrix $D$ is given by 
$$\det D=  \lambda \ \prod_{t=1}^{b}\cof D_t =\lambda \ \cof D, $$
where $\lambda = \sum_{t=1}^b \lambda_t  \mbox{ with }  \lambda_t =\dfrac{\det D_t}{\cof D_t}.$
\end{theorem}

Now we state and prove a result which gives the inverse of the distance matrix for weighted cactoid-type digraph.

\begin{theorem}
Let $G=(V,E)$ be a cactoid-type digraph with blocks $G_t$; $1\leq t\leq b$ and $W:E\to \mathbb{R}$ be a weight function. For $1\leq t\leq b$, let $D$ and $D_t$ be the   distance matrix of weighted digraphs $G^W$ and $G_t^W$, respectively.  If $\det D \neq 0$, $\det D_t \neq 0$ and $\cof D_t \neq 0$, for all $1\leq t\leq b$, then there exists a Laplacian-like matrix $\mathcal{L}$, column vectors $\alpha,\beta$ and a real number $\lambda$, such that  
$$D^{-1} = -\mathcal{L} + \dfrac{1}{\lambda}\  \beta \alpha^T.$$
\end{theorem}
\begin{proof}
For $1\leq t\leq b$, let $D_t$ be the distance matrix of the block $G_t^W$. Since a block $G_t^W$ is a weighted  digraph $dC(n;m_1,\cdots,m_r)^W$, so  by Theorem~\ref{thm:rank-one-single},  $D_t$ is a left LapExp($\lambda_t, \alpha_t, \beta_t, \mathcal{L}_t$) matrix. Therefore, using Theorem~\ref{thm:extn},  $D$ is a left LapExp($\lambda, \alpha, \beta, \mathcal{L}$) matrix. Further, by Theorem~\ref{lem:det-w-cactoid},   $\det D \neq 0$  implies that $\lambda \neq 0$ and hence by Lemma~\ref{lem:inv}, the result follows.  
\end{proof}

The next section deal with the determinant of the distance matrix for  undirected and unweighted graph  $C(n;m_1,\cdots,m_r)$.

\section{Determinant of the Distance Matrix for $C(n;m_1,m_2,\cdots,m_r)$}\label{sec:UnOriented Block}

Recall that, $C(n;m_1,\cdots,m_r)$ is an undirected graph consisting of $r$  cycles  of length $n+m_1+1,n+m_2+1,\ldots,n+m_r+1$ sharing a common path $P$ of length $n$ such that the intersection between any two cycles is  precisely the path $P$. In this section, we compute the determinant of the distance matrix for a class of $C(n;m_1,\cdots,m_r)$. 

In literature, similar problems has been studied for the cases $r=1,2$ (for details see~\cite{Bp2, Drat,Drat1,Gong, Hou2}) and we have shown that some of their results can be extended for $r\geq 3$. We  recall some  results for the case $r=1$, {\it i.e.},  cycles $C_n$; $n\geq 1$,  useful for the subsequent results. For cycles of even length $C_{2k}$, by~\cite[Theorem 3.4]{Bp2} we get $\det D(C_{2k})=0$. Next, for the odd 
cycle $C_{2k+1}$ the vertices are labeled so that the vertex $i$ is adjacent to vertices $i-1$ and $i+1$  (the indices are taken modulo $2k+1$). Let $C$ be the cyclic permutation matrix of order $2k+1$ such that $C_{i,i+1}=1$, $1\leq i\leq 2k+1$  (again taking indices modulo $2k+1$). The next result gives the determinant and the inverse of the distance matrix for odd cycles.
\begin{theorem}\label{thm:invese-odd-cycle}(\cite[Theorem 3.1]{Bp2} and~\cite[Lemma 2.1]{Hou2})
Let $D$ be the distance matrix of the odd cycle $C_{2k+1}$ on $2k+1$ vertices. Then,  $\det D= k(k+1)$ and the inverse is given by
$$D^{-1}= -2I - C^k - C^{k+1}+\frac{2k+1}{k(k+1)}J.$$
\end{theorem}

In Section~\ref{sec:Oriented Block},  we found that  $C(n;m_1,\cdots,m_r)$ can be identified with $r+1$ paths. Due to this identification, we will be able to show for a large class the determinant of distance matrix is zero. We will follow the same indexing of vertices and conventions, as in the  Section~\ref{sec:Oriented Block}, {\it i.e.},  the vertex set  of   $C(n;m_1,\cdots,m_r)$ is $\{u_i: \ 1\leq i\leq n\} \cup \{v_i^{(j)}: 1\leq i\leq m_j \ , \ 1\leq j\leq r\}$ with  $m_1\leq m_2\leq\cdots\leq m_r$.

Given an unweighted path $P$, if we draw a line through the middle of the path, then the following are true: $(a)$  if $P$ have even number of vertices, then vertex set is partitioned into two equal halves,  $(b)$   if $P$ have an odd number of vertices, then vertex set is partitioned into two equal halves and a middle vertex. Since $C(n;m_1,\cdots,m_r)$ is a planar graph and also can be identified with $r+1$ paths, so if we draw a line ``$\mathfrak{L}$" passing through the middle of  all the $r+1$ paths, then the vertex set is partitioned in the following:
%%\begin{itemizeReduced}
%%\item [{\it 1}.] The set of vertices in the upper half of the line $\mathfrak{L}$ is denoted by UH$\mathfrak{L}$.
%% \item [{\it 2}.]The set of vertices in the lower half of the line $\mathfrak{L}$ is denoted by LH$\mathfrak{L}$.
%% \item [{\it 3}.]Vertices on the line $\mathfrak{L}$ (non-empty only when the path contains an odd number of vertices).
%%\end{itemizeReduced}
\begin{itemize}
\item [{\it 1}.] The set of vertices in the upper half of the line $\mathfrak{L}$ is denoted by UH$\mathfrak{L}$.
 \item [{\it 2}.]The set of vertices in the lower half of the line $\mathfrak{L}$ is denoted by LH$\mathfrak{L}$.
 \item [{\it 3}.]Vertices on the line $\mathfrak{L}$ (non-empty only when the path contains an odd number of vertices).
\end{itemize}

 \begin{figure}[ht]
     \centering
     \begin{subfigure}{.30\textwidth}
         %\centering
         \begin{tikzpicture}[scale=0.6]
\Vertex[x=1,y=6,label=$u_{2k}$,position=left,size=.2,color=red]{1}
\Vertex[x=1,y=4,label=$u_{k+1}$,position=left,size=.2,color=red]{2}
\Vertex[x=1,y=3,label=$u_{k}$,position=left,size=.2,color=black]{3}
\Vertex[x=1,y=2,label=$u_{k-1}$,position=left,size=.2,color=red]{4}
\Vertex[x=1,y=0,label=$u_0$,position=left,size=.2,color=red]{5}
\Vertex[x=2,y=3,label=$v_1^{(1)}$,position=right,size=.2,color=black]{6}
\Vertex[x=3.5,y=3,label=$v_{1}^{(2)}$,position=right,size=.2,color=black]{7}
\Vertex[x=6,y=3,label=$v_{1}^{(r)}$,position=right,size=.2,color=black]{8}

\draw [dots]  (-1,3) -- (8,3);

\Edge[style={dashed}](1)(2)
\Edge(3)(2)
\Edge(3)(4)
\Edge[style={dashed}](4)(5)
\Edge(1)(6)
\Edge(5)(6)
\Edge(1)(7)
\Edge(5)(7)
\Edge(1)(8)
\Edge(5)(8)
\end{tikzpicture}

         \caption{$C(2k;1,1,\cdots,1)$}
     \end{subfigure}% <- nötig
       \begin{subfigure}{.35\textwidth}
         \centering
         \begin{tikzpicture}[scale=0.6]
\Vertex[x=1,y=6,label=$u_{2k}$,position=left,size=.2,color=red]{1}
\Vertex[x=1,y=4,label=$u_{k+1}$,position=left,size=.2,color=red]{2}
\Vertex[x=1,y=3,label=$u_{k}$,position=left,size=.2,color=black]{3}
\Vertex[x=1,y=2,label=$u_{k-1}$,position=left,size=.2,color=red]{4}
\Vertex[x=1,y=0,label=$u_0$,position=left,size=.2,color=red]{5}
\Vertex[x=2,y=3,label=$v_1^{(1)}$,position=above,size=.2,color=black]{6}
\Vertex[x=3,y=3,label=$v_{1}^{(2)}$,position=above,size=.2,color=black]{7}
\Vertex[x=6,y=3,label=$v_{1}^{(r-1)}$,position=above,size=.2,color=black]{8}
\Vertex[x=7,y=6,label=$v_1^{(r)}$,position=right,size=.2,color=black]{9}
\Vertex[x=7,y=4,label=$v_l^{(r)}$,position=right,size=.2,color=green]{10}
\Vertex[x=7,y=2,label=$v_{l+1}^{(r)}$,position=right,size=.2,color=green]{11}
\Vertex[x=7,y=0,label=$v_{2l}^{(r)}$,position=right,size=.2,color=black]{12}

\draw [dots]  (0,3) -- (9,3);

\Edge[style={dashed}](1)(2)
\Edge(3)(2)
\Edge(3)(4)
\Edge[style={dashed}](4)(5)
\Edge(1)(6)
\Edge(5)(6)
\Edge(1)(7)
\Edge(5)(7)
\Edge(1)(8)
\Edge(5)(8)
\Edge(1)(9)
\Edge(5)(12)
\Edge[style={dashed}](9)(10)
\Edge[style={dashed}](11)(12)
\Edge(10)(11)
\end{tikzpicture}
         \caption{$C(2k;1,1,\cdots,1,2l)$}
     \end{subfigure}
     \begin{subfigure}{.3\textwidth}
         \centering
       \begin{tikzpicture}[scale=0.6]
\Vertex[x=1,y=6,label=$u_{2k-1}$,position=right,size=.2,color=black]{1}
\Vertex[x=1,y=0,label=$u_0$,position=right,size=.2,color=black]{5}
\Vertex[x=2,y=3,label=$v_1^{(1)}$,position=right,size=.2,color=black]{6}
\Vertex[x=3.5,y=3,label=$v_{1}^{(2)}$,position=right,size=.2,color=black]{7}
\Vertex[x=6,y=3,label=$v_{1}^{(r)}$,position=right,size=.2,color=black]{8}

\draw [dots]  (-1,3) -- (7,3);

\Edge[style={dashed}](1)(5)
\Edge(1)(6)
\Edge(5)(6)
\Edge(1)(7)
\Edge(5)(7)
\Edge(1)(8)
\Edge(5)(8)
\end{tikzpicture}
       
         \caption{$C(2k-1;1,1,\cdots,1)$}
     \end{subfigure}
     \caption{}\label{fig:2}
\end{figure}

 We begin with the result which gives the determinant of the distance matrix, whenever the common path is of even length $\geq 4$.

 \begin{theorem}
Let $r\geq 2$ and  $n \geq 4$ with $n$ even. Let $D$ be the distance matrix of  $C(n;m_1,\cdots,m_r)$. Then $\det D = 0$.
\end{theorem}

\begin{proof}
 We will prove this result  case by case, by  dividing the problem in the following cases:
 \begin{itemizeReduced}
 \item [Case 1.] $m_j=1$, for all $1\leq j\leq r$.
 \item [Case 2.] For $1\leq j\leq r$, exactly one of the $m_j$'s is even and all other $m_j=1$. 
 \item [Case 3.] For $1\leq j\leq r$, at least two of $m_j$'s are even.
 \item [Case 4.]  For $1\leq j\leq r$, at least one of $m_j$'s (say $m_s$), is odd with  $m_s\geq 3$.
 \end{itemizeReduced}
 For $n \geq 4$ and $n$ is even, suppose $n=2k$ whenever $k\geq 2.$   Now  we provide the proofs.\\
 
\noindent \textbf{Case 1.} Since $m_j=1$; $1\leq j\leq r$, so $d(u_0,u_n)=2$, (see Figure~\ref{fig:2}$(a)$). For any vertex $y $ on the line  $ \mathfrak{L}$, we have  $d(y,u_0)-d(y,u_{k-1}) = d(y,u_n) - d(y,u_{k+1})$. Next, if $y$ is in UH$\mathfrak{L}$, then  $d(y,u_0) = d(y,u_n) + 2$ and $d(y,u_{k-1}) = d(y,u_{k+1}) + 2$. Similarly, if $y$ is in LH$\mathfrak{L}$, then $d(y,u_n) = d(y,u_0) + 2$ and $d(y,u_{k+1}) = d(y,u_{k-1}) + 2$. Thus $d(y,u_0)-d(y,u_{k-1}) + d(y,u_{k+1}) - d(y,u_n) = 0$ for every vertex $y$ in $C(n;m_1,\cdots,m_r)$.  It implies that,  the columns in $D$ corresponding to vertices $\{u_0,u_{k-1}, u_{k+1}, u_n\}$ are linearly dependent and hence $\det D=0$.\\

\noindent\textbf{Case 2.} Following the conventions, assume $m_j=1$; $1\leq j\leq r-1$ and $m_r=2l$, where $l\geq 1$. Since  $r\geq 2$, so   $d(u_0,u_n)=2$ (see Figure~\ref{fig:2}$(b)$). 

For any vertex $y $ on the line $ \mathfrak{L}$, we have $d(y,u_0) = d(y,u_n)$ and $d(y,u_{k-1}) = d(y,u_{k+1})$. Next, for $y$ in UH$\mathfrak{L}$, we have the following sub cases.  
\begin{itemizeReduced}
\item Whenever  $y\neq v_l^{(r)}$, we have $d(y,u_0) = d(y,u_n)+d(u_n, u_0)= d(y,u_n) + 2$ and similar arguments also leads to $d(y,u_{k+1}) = d(y,u_{k-1}) + 2$.

\item Whenever $y= v_l^{(r)}$, we have $d(y,u_0) =d(y,u_n) + 1$ and $d(y,u_{k-1}) = d(y,u_{k+1}) + 1$.
\end{itemizeReduced}
Similarly, for $y$ in LH$\mathfrak{L}$:\vspace*{-.2cm}
\begin{itemizeReduced}
\item If $y\neq v_{l+1}^{(r)}$, then $d(y,u_n) = d(y,u_0) + 2$ and $d(y,u_{k+1}) = d(y,u_{k-1}) + 2$.

\item If  $y= v_{l+1}^{(r)}$, then $d(y,u_n)= d(y,u_0) + 1$ and $d(y,u_{k+1}) =d(y,u_{k-1}) + 1$.
\end{itemizeReduced} Thus, $d(y,u_0)-d(y,u_{k-1}) + d(y,u_{k+1}) - d(y,u_n) = 0$,  for every vertex $y$. Hence, the columns in $D$ corresponding to vertices $\{u_0,u_{k-1}, u_{k+1}, u_n\}$ are linearly dependent and the results holds good for this case.\\

\noindent\textbf{Case 3.} 
 For $1\leq s,t\leq r$, let us assume $m_s=2l$ and $m_t=2p$, where $l,p\geq 1$. If a  vertex $y$ is on the line $ \mathfrak{L}$, then  $ d(y,v_{l+1}^{(s)})=  d(y,v_{l}^{(s)}) $ and  $d(y,v_{p+1}^{(t)})=d(y,v_{p}^{(t)})$. For  $y$ in UH$\mathfrak{L}$, we have $ d(y,v_{l+1}^{(s)})=  d(y,v_{l}^{(s)}) +1$ and  $d(y,v_{p+1}^{(t)})=d(y,v_{p}^{(t)})+1$. Similarly, for $y$ in LH$\mathfrak{L}$, we get $ d(y,v_{l}^{(s)})=  d(y,v_{l+1}^{(s)})+1 $ and  $d(y,v_{p}^{(t)})=d(y,v_{p+1}^{(t)})+1$. Thus, $ d(y,v_{l}^{(s)})-  d(y,v_{l+1}^{(s)})- d(y,v_{p}^{(t)}) + d(y,v_{p+1}^{(t)})=0$,  for every vertex $y$. Hence, the columns in $D$ corresponding to vertices $\{v_{l}^{(s)}, v_{l+1}^{(s)}, v_{p}^{(t)}, v_{p+1}^{(t)} \}$ are linearly dependent and the result follows.\\
  
\noindent\textbf{Case 4.}  Assume that, $m_s=2l+1$, where $l\geq 1$. If a  vertex $y$ is on the line $ \mathfrak{L}$, then $d(y,u_{k-1}) = d(y,u_{k+1})$ and $d(y,v_{l}^{(s)}) = d(y,v_{l+2}^{(s)})$. Next, for  $y$ in UH$\mathfrak{L}$, we have  $ d(y,v_{l+2}^{(s)}) = d(y,v_{l}^{(s)}) + 2$ and $d(y,u_{k-1}) = d(y,u_{k+1}) + 2$. Finally, for  $y$ in LH$\mathfrak{L}$, we get $d(y,v_{l}^{(s)}) = d(y,v_{l+2}^{(s)}) + 2$ and $d(y,u_{k+1}) = d(y,u_{k-1}) + 2$. Thus, for all vertices $y$, we have $d(y,u_{k-1}) - d(y,u_{k+1})+d(y,v_{l}^{(s)}) -d(y,v_{l+2}^{(s)})= 0$, which implies that, the columns of $D$ corresponding to vertices $\{u_{k-1}, u_{k+1}, v_{l}^{(s)}, v_{l+2}^{(s)}\}$ are linearly dependent. Hence $\det D=0$. This completes the proof.
\end{proof}

 Now we consider the cases whenever the common path is of odd length $\geq 3$, and we begin with the following lemma.
 
\begin{lem}\label{lem:odd-cycle}
Let $r\geq 2$,  $n \geq 3$ and  $n$  odd.  Let $D$ be the distance matrix of  $C(n;m_1,\cdots,m_r)$. If $n=2k-1$, where $k\ge 2$ and $m_j=1$, where $1\leq j\leq r$, then the determinant of $D$ is given by   
$$\det D = (-2)^{r-1} \left[k(k+1) - (r-1)(3k^2-k-2) \right].$$
\end{lem}
\begin{proof}
Let  $D_{11}$ be the distance matrix of the odd cycle of length $2k+1$, the induced subgraph generated by the vertices $\{u_0,u_1,\ldots, u_n, v_1^{(1)}\}$ in  $C(n;m_1,\cdots,m_r)$, see Figure~\ref{fig:2}$(c)$. We have partitioned the distance matrix of  $C(n;m_1,\cdots,m_r)$ as below: 
$$D = 
\left[
\begin{array}{c|c}
D_{11}& D_{12} \\
\hline
D_{21} & D_{22} \\
\end{array}
\right].$$
For $1\leq j\leq r$, let $R_{v_1}^{(j)}/ C_{v_1}^{(j)}$ denote the row/column of $D$ corresponding to the vertex $v_1^{(j)}$. With this  observation, all the rows and columns of $D$ corresponding the vertices  $\{v_1^{(j)}: \ 1\leq j\leq r\}$ are equal, except at the diagonal entries. So,  we begin with the   elementary column operations on $D$,  
  $C_{v_1}^{(j)} \leftarrow C_{v_1}^{(j)} - C_{v_1}^{(1)}$, for all $2\leq j\leq r$,  followed by the row operations  $R_{v_1}^{(j)} \leftarrow R_{v_1}^{(j)} - R_{v_1}^{(1)}$, for all $2\leq j\leq r$. Then the block form of the resulting matrix is given by,
$$\ds\widetilde{D} = 
\left[
\begin{array}{c|c}
  D_{11}& \widetilde{D}_{12} \\
\midrule
\widetilde{D}_{21} &\widetilde{ D}_{22} \\
\end{array}
\right],$$
 where $\widetilde{ D}_{12} = \left[
\begin{array}{c}
\mathbf{0}_{2k \times (r-1)}\\
\hline
2 \mathds{1}_{1 \times (r-1)}\\
\end{array}
\right] = \widetilde{ D}_{21}^T$ and $\widetilde{ D}_{22} = -2(I_{r-1}+J_{r-1})$. 

 Using Theorem~\ref{thm:invese-odd-cycle},  $D_{11}$ is invertible and the inverse is given by $$D_{11}^{-1} = -2I - C^k - C^{k+1}+\frac{2k+1}{k(k+1)}J,\mbox{ where  }
C^k_{ij} =
\begin{cases}
 1 & \text{if } j = [i+k] \ (mod \ 2k+1),\\
 0 & \text{otherwise}.
\end{cases}
$$ 
It is easy to see, the only non-zero row of $C^k \widetilde{D}_{12}$ is  the $(k+1)$th row, equals to $2 \mathds{1}^T$. Similarly, for $C^{k+1} \widetilde{D}_{12}$ the $k$th row is $2 \mathds{1}^T$ and remaining are zero. Thus  $\widetilde{D}_{21}D_{11}^{-1}\widetilde{D}_{12} = 4 \left( \frac{2k+1}{k(k+1)} -2\right)J_{r-1}$. Using Proposition~\ref{prop:blockdet}, we have
\begin{align*}
 \det \widetilde{D}& = \det D_{11} \times \det(\widetilde{D}_{22} -\widetilde{D}_{21}D_{11}^{-1}\widetilde{D}_{12})\\
	   & = k(k+1) \det\left(-2(I_{r-1}+J_{r-1}) - 4 \left( \frac{2k+1}{k(k+1)} -2\right)J_{r-1} \right)\\
	   & = k(k+1) \det\left(-2I_{r-1} + \left(6- \frac{4(2k+1)}{k(k+1)} \right)J_{r-1} \right)\\
	   & =(-2)^{r-1} \left[k(k+1) - (r-1)(3k^2-k-2)\right]
\end{align*}
and hence the result follows.
\end{proof}
The next theorem gives the determinant of the distance matrix, whenever the common path is of odd length $\geq 3.$

\begin{theorem}
Let $r\geq 2$,  $n \geq 3$ and  $n$  odd.  Let $D$ be the distance matrix of the  $C(n;m_1,\cdots,m_r)$.  Then
\begin{equation*}
\ds \det D = 
\begin{cases}
(-2)^{r-1} \left[k(k+1) - (r-1)(3k^2-k-2) \right] & \textup{ if }  n=2k-1, \ m_j = 1; 1\leq j\leq r,\\
0 & \textup{ otherwise. }
\end{cases}
\end{equation*} 
\end{theorem}

\begin{proof}
In view of Lemma~\ref{lem:odd-cycle}, it is enough to establish the result for the following cases:
 \begin{itemizeReduced}
 \item [Case 1.] For $1\leq j\leq r$, exactly one of the $m_j$'s is odd  (say $m_s$),  with $m_s\geq 3$ and all other $m_j=1$.
 \item [Case 2.] For $1\leq j\leq r$, at least two of $m_j$'s are odd ($\geq 3$). 
 \item [Case 3.]For $1\leq j\leq r$, at least one of $m_j$'s  is even.
  \end{itemizeReduced}
  
\noindent\textbf{Case 1.} Suppose $m_s=2l+1$, where $l\geq 3$.  Note that $d(u_0,u_n)=2$. If a  vertex $y $ is on the line $ \mathfrak{L}$, then $d(y,u_0)=d(y,u_n)$ and $d(y,v_{l+2}^{(s)})=d(y,v_{l}^{(s)})$. For  $y$ in UH$\mathfrak{L}$, we have  $d(y,u_0)=d(y,u_n)+2$ and $d(y,v_{l+2}^{(s)})=d(y,v_{l}^{(s)})+2$. Finally, for  $y$ in LH$\mathfrak{L}$, we get $d(y,u_n)=d(y,u_0)+2$ and 
$d(y,v_{l}^{(s)})=d(y,v_{l+2}^{(s)})+2$. Thus,  $d(y,u_0)-d(y,u_n)+d(y,v_{l}^{(s)})-d(y,v_{l+2}^{(s)})=0$ for every vertex $y$. Hence $\det D=0$.\\

\noindent\textbf{Case 2.} 
 For $1\leq s,t\leq r$, let us assume $m_s=2l$ and $m_t=2p$, where $l,p \geq 1$. If a  vertex $y$ is on the line $ \mathfrak{L}$, then $d(y, v_{l}^{(s)} )=d(y,v_{l+2}^{(s)})$ and  $d(y, v_{p}^{(t)} )=d(y,v_{p+2}^{(t)})$. Next, for $y$ in UH$\mathfrak{L}$, we have $d(y,v_{l+2}^{(s)})=d(y,v_{l}^{(s)})+2$ and $d(y,v_{p+2}^{(t)})=d(y,v_{p}^{(t)})+2$. Similarly, for $y$ in LH$\mathfrak{L}$, we get $d(y, v_{l}^{(s)} )=d(y,v_{l+2}^{(s)})+2$ and  $d(y, v_{p}^{(t)} )=d(y,v_{p+2}^{(t)})+2$. Thus, $d(y, v_{l}^{(s)} )-d(y,v_{l+2}^{(s)})-d(y, v_{p}^{(t)} )+d(y,v_{p+2}^{(t)})=0$ for every vertex $y$ and the result holds true for this case.\\
 
 \noindent\textbf{Case 3.} Since $n\geq 3$ and $n$  odd, so let $n=2k+1$ where $k\geq 1$. Also assume $m_s=2l, $ where $l\geq 1$. For this case   we will show columns corresponding to vertices $\{u_k,u_{k+1}, v_{l}^{(s)},v_{l+1}^{(s)}\}$ are linearly dependent. If a  vertex $y $ is on the line $ \mathfrak{L}$, then $d(y, u_k)=d(y, u_{k+1})$ and $d(y, v_{l}^{(s)})=d(y, v_{l+1}^{(s)})$. for $y$ in UH$\mathfrak{L}$, we have $d(y, u_{k})=d(y, u_{k+1})+1$ and $d(y, v_{l+1}^{(s)})=d(y, v_{l}^{(s)})+1$. Finally, for $y$ in LH$\mathfrak{L}$, we have $d(y, u_{k+1})=d(y, u_{k})+1$ and $d(y, v_{l}^{(s)})=d(y, v_{l+1}^{(s)})+1$. Thus $d(y, u_k)-d(y, u_{k+1})+ d(y, v_{l}^{(s)})-d(y, v_{l+1}^{(s)})=0$ for every vertex $v$ and hence the result follows.
 \end{proof}

 Next we find the determinant of the distance matrix whenever the common path is of length $n=2$. We begin with a result in which the idea of the proof is similar to the proof of Lemma~\ref{lem:odd-cycle}, but for the sake of completeness we provide the details.
\begin{lem}\label{lem:n=2,odd}
Let $r \geq 2$ and $n = 2$. Let $D$ be the distance matrix of  $C(n;m_1,\cdots,m_r)$. If  $m_j=1$;  $1\leq j\leq r-1$ and $m_r=2(k-1)$; $k\ge 2$, then the determinant of $D$ is given by   
$$\det D = (-2)^{r-1} \left[k(k+1) - (r-1)(3k^2-k-2) \right].$$
\end{lem} 

\begin{proof}
Let  $D_{11}$ be the distance matrix of the odd cycle of length $2k+1$, the induced subgraph generated by the vertices $\{u_0,u_1, u_2, v_1^{(r)},v_2^{(r)},\ldots,v_{m_r}^{(r)}\}$, where $m_r=2(k-1)$; $k\ge 2$ in  $C(n;m_1,\cdots,m_r)$. We have partitioned the distance matrix of  $C(n;m_1,\cdots,m_r)$ as below: 
$$D = 
\left[
\begin{array}{c|c}
D_{11}& D_{12} \\
\hline
D_{21} & D_{22} \\
\end{array}
\right].$$

Let $R_{u_1}/C_{u_1}$ and $R_{v_1}^{(j)}/ C_{v_1}^{(j)}$ denote the row/column of $D$ corresponding to the vertex $u_1$ and $v_1^{(j)}$; $1\leq j\leq r-1$, respectively. Observe that, for each $1\leq j\leq r-1$, the distance between vertices $\{v_1^{(r)},v_2^{(r)},\ldots,v_{m_r}^{(r)}\}$  and   $v_1^{(j)}$ is  same as with the vertex $u_1$. So, we begin with the following elementary column operations on $D$,   $C_{v_1}^{(j)} \leftarrow C_{v_1}^{(j)} - C_{u_1}$, for all $1\leq j\leq r-1$,  followed by the row operations  $R_{v_1}^{(j)} \leftarrow R_{v_1}^{(j)} - R_{u_1}$, for all $1\leq j\leq r-1$. Then the block form of the resulting matrix is given by
$$\ds\widetilde{D} = 
\left[
\begin{array}{c|c}
  D_{11}& \widetilde{D}_{12} \\
\midrule
\widetilde{D}_{21} &\widetilde{ D}_{22} \\
\end{array}
\right],$$
 where $\widetilde{ D}_{12} = \left[
\begin{array}{c}
\mathbf{0}_{1 \times (r-1)}\\
\hline
2 \mathds{1}_{1 \times (r-1)}\\
\hline
\mathbf{0}_{(2k-1) \times (r-1)}\\
\end{array}
\right] = \widetilde{ D}_{21}^T$ and $\widetilde{ D}_{22} = -2(I_{r-1}+J_{r-1})$. Using Theorem~\ref{thm:invese-odd-cycle},  $D_{11}$ is invertible and the inverse is given by $$D_{11}^{-1} = -2I - C^k - C^{k+1}+\frac{2k+1}{k(k+1)}J,\mbox{ where  }
C^k_{ij} =
\begin{cases}
 1 & \text{if } j = [i+k] \ (mod \ 2k+1),\\
 0 & \text{otherwise}
\end{cases}
$$ 
and hence  $\widetilde{D}_{21}D_{11}^{-1}\widetilde{D}_{12} = 4 \left( \frac{2k+1}{k(k+1)} -2\right)J_{r-1}$. Using Proposition~\ref{prop:blockdet}, we have
\begin{align*}
 \det \widetilde{D}& = \det D_{11} \times \det(\widetilde{D}_{22} -\widetilde{D}_{21}D_{11}^{-1}\widetilde{D}_{12})\\
	   & = k(k+1) \det\left(-2(I_{r-1}+J_{r-1}) - 4 \left( \frac{2k+1}{k(k+1)} -2\right)J_{r-1} \right)\\
	   & = k(k+1) \det\left(-2I_{r-1} + \left(6- \frac{4(2k+1)}{k(k+1)} \right)J_{r-1} \right)\\
	   & =(-2)^{r-1} \left[k(k+1) - (r-1)(3k^2-k-2)\right]
\end{align*}
and hence the result follows.
\end{proof}

\begin{theorem}
Let $r \geq 2, n=2$ and  $D$ be the distance matrix of  $C(n;m_1,\cdots,m_r)$. Then
\begin{equation*}
\ds \det D = 
\begin{cases}
 (-1)^{r+1} \ 2^{r+2}(r-1) & \textup{ if }  m_j=1; 1\leq j\leq r,\\
 \\
(-2)^{r-1} \left[k(k+1) - (r-1)(3k^2-k-2) \right] & \textup{ if }  m_j = 1; 1\leq j\leq r-1 \textup{ and } \\
& \ \quad m_r=2(k-1); k\geq 2, \\
0 & \textup{ otherwise. }
\end{cases}
\end{equation*} 
\end{theorem}
\begin{proof}
Observe that, for $n=2$ with $m_j=1$; $1\leq j\leq r$, the graph is a complete bipartite graph $K_{2, r+1}$, with the following vertex partition $\{u_0,u_2\}$ and $\{u_1, v_{1}^{(j)}: \ 1\leq j\leq r\}$. Thus, by~\cite{Hou3} $\det D=(-1)^{r+1} \ 2^{r+2}(r-1)$. Next, in view of Lemma~\ref{lem:n=2,odd}, it is enough to establish the result for the following two cases:
 \begin{itemizeReduced}
 \item [Case 1.] For $1\leq j\leq r$, one of the $m_j$'s (say, $m_s$)  is odd, with $m_s\geq 3$.
 \item [Case 2.] For $1\leq j\leq r$, at least two of $m_j$'s are even. 
  \end{itemizeReduced}

\noindent\textbf{Case 1.} Assume that $m_s=2l+1$, where $l\geq 1$. If a  vertex $y $ is on the line $ \mathfrak{L}$, then $d(y,u_0) = d(y,u_2)$ and $d(y,v_{l}^{(s)}) = d(y,v_{l+2}^{(s)})$. For $y$ in UH$\mathfrak{L}$, we have  $d(y,u_0) = d(y,u_2)+2$ and $d(y,v_{l+2}^{(s)}) = d(y,v_l^{(s)})+2$. Finally, for $y$ in LH$\mathfrak{L}$, we have $d(y,u_2) = d(y,u_0)+2$ and $d(y,v_{l}^{(s)}) = d(y,v_{l+2}^{(s)})+2$. Hence  $d(y,u_0) - d(y,u_2)+ d(y,v_{l}^{(s)}) - d(y,v_{l+2}^{(s)})=0$, for every vertex $y$ and the result follows.\\

\noindent\textbf{Case 2.}  Let $1\leq s,t\leq r$.   Assume that $m_s=2l$ and $m_t=2p$, where $l,p \geq 1$.  If a  vertex $y $ is on the line $ \mathfrak{L}$, then  $d(y, v_{l}^{(s)} )=d(y,v_{l+1}^{(s)})$ and  $d(y, v_{p}^{(t)} )=d(y,v_{p+1}^{(t)})$. If a  vertex $y$ is in UH$\mathfrak{L}$, we have $d(y,v_{l+1}^{(s)})=d(y,v_{l}^{(s)})+1$ and $d(y,v_{p+1}^{(t)})=d(y,v_{p}^{(t)})+1$. Similarly, for $y$ in LH$\mathfrak{L}$, we get $d(y, v_{l}^{(s)} )=d(y,v_{l+1}^{(s)})+1$ and  $d(y, v_{p}^{(t)} )=d(y,v_{p+1}^{(t)})+1$. Thus, $d(y, v_{l}^{(s)} )-d(y,v_{l+1}^{(s)})-d(y, v_{p}^{(t)} )+d(y,v_{p+1}^{(t)})=0$ for every vertex $y$ and the result holds true for this case. 
\end{proof}
 
Finally, we compute the determinant of the distance matrix for a few cases whenever the common path is of length $n=1$. 

\begin{theorem}
Let $n=1$ and  $D$ be the distance matrix of   $C(n;m_1,\cdots,m_r)$. Then the following holds:
\begin{itemize}
\item[(i)] If $m_j=1$; $1\leq j\leq r$, then $\det D=(-1)^{r-1} \ 2^{r-2}$. 

\item[(ii)] If  one of the $m_j$'s  is even, $1\leq j\leq r$, then $\det D=0$.
\end{itemize}
\end{theorem}
\begin{proof}
In literature, the  graph $C(n;m_1,\cdots,m_r)$ with  $n=1$ and $m_j=1$; $1\leq j\leq r$, is denoted by $T_r$ and in~\cite{JD}, it was shown that $\det D(T_r)=(-1)^{r-1} \ 2^{r-2}.$  To prove $(ii)$, we assume $m_s=2l$, where $l\geq 1$.  If a  vertex $y $ is on the line $ \mathfrak{L}$, then $d(y,u_0) = d(y,u_1)$ and $d(y,v_{l+1}^{(s)}) = d(y,v_l^{(s)})$. Next, for $y$ in UH$\mathfrak{L}$, we have  $d(y,u_0) = d(y,u_1)+1$ and $d(y,v_{l+1}^{(s)}) = d(y,v_l^{(s)})+1$. Similarly, for $y$ in LH$\mathfrak{L}$, we have $d(y,u_1) = d(y,u_0)+1$ and $d(y,v_{l}^{(s)}) = d(y,v_{l+1}^{(s)})+1$. Hence  $d(y,u_1) - d(y,u_0)- d(y,v_{l}^{(s)}) +d(y,v_{l+1}^{(s)})=0$, for every vertex $y$ and the result follows.
\end{proof}

\begin{rem}
The cases when $n=1$ with $m_j \geq 3$ is odd has been computationally seen to have non-zero determinant, but there is no constructive proof yet.
\end{rem}

\section{Conclusion}
In this article, we first compute the determinant and cofactor of the distance matrix for weighted digraphs consisting of finitely many oriented cycles, that shares a common directed path.  If the distance matrix is invertible, then we have shown the inverse as a rank one perturbation of a multiple of the Laplacian-like matrix similar to trees. In this search, we found the Laplacian-like matrix is a perturbed weighted Laplacian for the digraph. Interestingly, the adjacency matrix involved in the weighted Laplacian for the digraph is not dependent on the weights on the individual edge, but only on the total weight of each of the cycles involved.  It allows us to choose the weight on edge to be $0$, without disturbing the structure.

Further, we compute the determinant of the distance matrix for weighted cactoid-type digraphs, if the cofactor of each block is non zero and find its inverse, whenever it exists. We also compute the determinant of the distance matrix for a class of graphs consisting of finitely many cycles, sharing a common path.\\

\noindent{ \textbf{\Large Acknowledgements}}: We take this opportunity to thank the anonymous reviewers for their guidance and suggestions which have immensely helped us in getting the article to its present form.  Sumit Mohanty would like to thank the Department of Science and Technology, India, for financial support through the projects MATRICS (MTR/2017/000458).

\end{document}